\documentclass[12pt]{article}
\usepackage[utf8]{inputenc}
\usepackage{amsthm}
\usepackage{amssymb}
\usepackage{multirow}
\usepackage{indentfirst}
\usepackage{amsmath}
\usepackage{dirtytalk}
\usepackage{breqn}
\usepackage{xcolor}
\usepackage{tikz-cd}
\usepackage{float}
\usepackage{lipsum}
\usepackage{moresize}
\usepackage[hyphens]{url}
\usepackage[colorlinks,pdfusetitle]{hyperref}
\usepackage[hyphenbreaks]{breakurl}
\usepackage[
  backend=biber,
  style=alphabetic,
  citestyle=alphabetic,
  backref=true]{biblatex}
\usepackage{geometry}
\usepackage{mathrsfs}
\newtheorem{theorem}{Theorem}[section]
\newtheorem{corollary}[theorem]{Corollary}
\newtheorem{remark}[theorem]{Remark}
\newtheorem{example}[theorem]{Example}
\newtheorem{lemma}[theorem]{Lemma}
\newtheorem{proposition}[theorem]{Proposition}
\newtheorem{conjecture}[theorem]{Conjecture}
\newtheorem*{corollary*}{Corollary}
\newtheorem*{theorem*}{Theorem}
\newtheorem*{remark*}{Remark}
\newtheorem*{example*}{Example}
\newtheorem*{lemma*}{Lemma}
\newtheorem*{axiom*}{Axiom}
\newtheorem*{proposition*}{Proposition}
\newtheorem*{conjecture*}{Conjecture}

\theoremstyle{definition}
\newtheorem{definition}[theorem]{Definition}
\newtheorem*{definition*}{Definition}

\newcommand{\R}{\mathbb{R}}
\newcommand{\C}{\mathbb{C}}

\newcommand{\Z}{\mathbb{Z}}

\newcommand{\g}{\mathfrak{g}}
\newcommand{\p}{\mathfrak{p}}
\renewcommand{\o}{\mathsf{O}}
\renewcommand{\O}{\mathcal{O}}
\newcommand{\n}{\mathcal{N}}
\newcommand{\SF}{\mathcal{S}}

\newcommand{\Pin}{\text{Pin}}

\newcommand{\Hom}{\text{Hom}}
\newcommand{\GL}{\text{GL}}

\newcommand{\Spec}{\text{Spec}}

\newcommand{\Tw} {\widetilde{W}}
\newcommand{\Tsig}
{\widetilde{\sigma}}
\newcommand{\Min}{\text{Min}}
\newcommand{\Irr}{\text{Irr}}
\newcommand{\Irrg}{\text{Irr}_{\text{\footnotesize gen}}}
\newcommand{\Ind}{\text{Ind}}
\newcommand{\Res}{\text{Res}}
\newcommand{\sgn}{\text{\sffamily sgn}}
\newcommand{\triv}{\text{\sffamily triv}}
\newcommand{\gen}{\text{gen}}
\newcommand{\sol}{\text{sol}}
\title{Spin Representations of Finite Coxeter Groups and Generalisations of Saxl's Conjecture}

\author{
Yutong Chen\thanks{Centre for Mathematical Sciences, University of Cambridge, Cambridge CB3 0WA, UK} 
\and Felix Gu\thanks{Mathematical Institute, University of Oxford, Oxford OX2 6GG, UK} 
\and Will Osborne\footnotemark[2]}

\date{}

\begin{document}

\maketitle

\begin{abstract}
This paper presents a natural generalisation of Saxl’s conjecture from a Lie-theoretical perspective, which is verified for the exceptional types. For classical types, progress is made using spin representations, revealing connections to certain tensor product decomposition problems in symmetric groups. We provide an alternative uniform description of the cuspidal family (in the sense of Lusztig) through spin representations, offering an equivalent conjecture formulation. Additionally, we generalise Saxl’s conjecture to finite Coxeter groups and prove it for the non-crystallographic cases.
\end{abstract}

\section{Introduction}
It is an open problem in representation theory to decompose tensor products of irreducible characters of the symmetric group $S_n$. A well-known open conjecture of this form is the following introduced by Saxl:

\begin{conjecture*}[Saxl]
    Let $\lambda=(k,k-1,...,2,1)$ be the staircase partition of $n=k(k+1)/2$, and let $\sigma_\lambda$ be the irreducible representation of $S_n$ indexed by $\lambda$. Then $\sigma_\lambda \otimes \sigma_\lambda$ contains all the irreducible representations of $S_n$ as constituents.
\end{conjecture*}

Recently, a lot of research has been done to identify many constituents of the tensor product (e.g. see \cite{BessenrodtSaxlConjecture} \cite{ikenmeyer2015saxl} \cite{pak2016kronecker}). While previous methods used to attack Saxl's conjecture are purely combinatorial, we shed new light on it by viewing $S_n$ as a Weyl group of type $A$ and considering the Lie-theoretic perspective. As for the irreducible character $\sigma_\lambda$ corresponding to the partition $\lambda=(k,k-1,...,2,1)$, we view $\lambda$ as the the unique \textbf{self-dual solvable} nilpotent orbit of $\mathfrak{sl}_n$. Self-duality in this case means $\lambda=\lambda^t$, while solvable means any element in the orbit has solvable centralizer in the Lie algebra. The alignment between irreducible characters of $S_n$ and the nilpotent orbits of $\mathfrak{sl}_n$ is a special case of the Springer correspondence.\\

In this setting, we formulate a natural generalisation of Saxl's conjecture to arbitrary Weyl groups. Fix a Weyl group $W$, its associated Lie algebra $\g$ and Lie group $G$, the set of nilpotent orbits $G \backslash \mathcal{N}$ in $\g$, and the set of irreducible representations $\Irr(W)$ of $W$. In \cite{Spaltensteindmapbook}, Spaltenstein defined an order-reversing map $d: G \backslash \mathcal{N} \to G \backslash \mathcal{N}$ that is an involution on its range (i.e. $d=d^3$), which is a natural generalisation of the transpose map for type A. If $\mathcal{O} \in G \backslash \mathcal{N}$ is in the range of $d$, then we say it is \textit{special}, and in this case we also say the $\sigma \in \Irr(W)$ attached to $\mathcal{O}$ and a trivial local system, via the Springer correspondence, are special.\\

On the other hand, Lusztig defined a partition of $\Irr(W)$ into so-called Lusztig families, which play an essential role in the classification of representations of finite groups of Lie type (see section 13.2 of \cite{carter1993finite}). For symmetric groups, each character forms a family on its own, while the partition is much coarser in other types. We may assign a Lusztig family to each special orbit $\O$ by the following procedure. Let $\sigma \in \Irr(W)$ be a corresponding special character as above. Say $\sigma$ is in Lusztig family $F$. Then $\O$ is the unique special orbit attached to a character in $F$, and we write $\O = \O_{F}$.\\ 

We're now ready to state our generalisation of Saxl's conjecture: 

\begin{conjecture*}[Generalised Saxl Conjecture for Weyl Groups]
    Let $\mathfrak{g}$ be a simple Lie algebra with Weyl group $W$. Suppose $\mathfrak{g}$ has a solvable nilpotent orbit $\O$ corresponding to the special character $\sigma_F=\sigma_{\O,\triv}$ in Lusztig family $F$, with the property that $d(\O)=\O$. Then $(\bigoplus_{\sigma\in F} \sigma)^{\otimes 2} $ contains all irreducible characters of $W$.
\end{conjecture*}

A solvable orbit satisfying the property $\O=d(\O)$ is called \textbf{self-dual}. The self-dual orbits are closely connected with the notion of of \say{cuspidal family} introduced by Lusztig \cite{lusztig1984characters} (explained in Remark \ref{cuspidal}).

Our main tool for attacking this conjecture is to use the projective representations of $W$. For each Weyl group $W$, there is a canonical degree 2 central extension $\Tw$, called the Pin cover of $W$. It is well-known that $\Irr(\Tw)$ is in 1-1 correspondence with the projective representations of $W$. Let $\Irrg(\Tw)$ be the set of all irreducible representations of $\Tw$ not inflated from $W$, we call these the \textbf{genuine} irreducible representations. Among $\Irrg(\Tw)$, there are 1 or 2 distinguished representations, called the spinor module(s). Ciubotaru \cite{ciubotaru2012spin} introduced a parameterisation of projective genuine representations of $W$ by the solvable orbits of $\g$ via a surjection: $$ \Psi: \Irr_{\gen}(\Tw) \twoheadrightarrow G \backslash \mathcal{N}_{\sol}$$ It has been proved that this parameterisation interacts nicely with the Springer correspondence when considering the tensor product decomposition of any genuine representation with the spinor module(s).\\

With the viewpoint of spin representations, we prove a theorem giving an alternative uniform description of the cuspidal family, see Theorem \ref{getallfamily}. This inspires the following equivalent version of the generalised Saxl conjecture.

\begin{conjecture*}
    Let $\mathfrak{g}$ be a simple Lie algebra with Weyl group $W$. Suppose $\mathfrak{g}$ has a solvable nilpotent orbit $\O$ such that $d(\O)=\O$. Then $$\bigwedge\nolimits^{\bullet}(V) \otimes \Big( \bigoplus_{\Tsig\in\Psi^{-1}(\O)} \Tsig \Big)^{\otimes 2}$$ contains all irreducible characters of $W$, where $V$ is the reflection representation of $W$.
\end{conjecture*}

If $\mathfrak{g}$ is of classical type, then $\Psi^{-1}(\O)$ contains either only one element, or two elements that are $\sgn$-dual to each other. However, $\bigwedge^{\bullet}(V)$ is $\sgn$-dual to itself. We therefore only need to look at $\bigwedge^{\bullet}(V)\otimes\Tsig\otimes\Tsig$ for any $\Tsig\in\Psi^{-1}(\O)$, since including $\Tsig \otimes \sgn$ in the tensor product provides only the same irreducible representations in the decomposition.\\

In this paper, we actually work towards a slightly stronger version of the generalised Saxl conjecture, see Conjecture \ref{gen_Saxl_conj_ver_I}. We prove this conjecture for all exceptional types, notably for $E_7$ when the family only contains two characters $\phi_{512,11},\phi_{512,12}$ and the statement mimics the \say{actual} Saxl conjecture. For the classical types, we prove partial results case by case using combinatorial methods. We also relate the generalised Saxl conjecture for types C and D to certain \say{tensor product decomposition} problems for symmetric groups, which roughly says the square tensor of a square (or almost square) partition contains \say{almost all} irreducible characters of $S_n$ for appropriate $n$ (note that it cannot contain all irreducible characters, by numerical experiments). For a precise statement, see Corollary \ref{tensorproddecompose}.\\

Eventually, we state a version of the generalised Saxl conjecture for all Coxeter groups (including non-crystallographic groups) dependent on the notion of \say{good families}, see Definition \ref{good}. This definition is natural in the sense of \cite{chan2013spin}. We then show this formulation is compatible with previous generalisations of the Saxl conjecture for Weyl groups, and prove it in the case of non-crystallographic Coxeter groups.

\subsection*{Acknowledgements}

We owe many thanks to our summer project supervisor Prof. Dan Ciubotaru for his invaluable mentorship over the summer. Dan guided us through the theory of spin representations, met with us regularly to discuss various ideas, directed us to many helpful resources, and gave us the initial motivation for this paper. We are grateful for the time spent sharing his great knowledge and precious advice throughout.

We would also like to thank the University of Oxford for providing partial funding for this project. In addition, the first author thanks the LMS Undergraduate Research Bursary, the second author thanks St Peter's College for the Simpson Fund Award, and the third author thanks the Crankstart Scholarship and Martingale Foundation, for their financial support.

\newpage
\section{Preliminaries}

\subsection{Clifford Algebras and Pin Groups}
More details on the following can be found in \cite{Meinrenken}, though we give all the necessary preliminaries here. Let $V$ be a vector space over a field $\C$, with a symmetric bilinear form $B:V \times V \to \C$. The \textbf{Clifford algebra} $C(V)$ is defined to be the quotient of the tensor algebra $T(V)$ by the ideal generated by:
$$ v \otimes v'+v' \otimes v+2B(v,v')=0, \; v, v' \in V $$

$C(V)$ is in fact a $\Z/2\Z$-graded algebra: $C(V)=C(V)^{[0]}\bigoplus C(V)^{[1]}$. This $\Z/2\Z$-grading inherits the $\Z$-grading on the tensor algebra $T(V)$: the even part $C(V)^{[0]}$ is generated by products of the form $v_1...v_k$ for $k$ even, whereas the odd part $C(V)^{[1]}$ is generated by products of the form $v_1...v_k$ for $k$ odd.

The representation theory of the Clifford algebra is particularly easy: it has only one or two irreducible representations, depending on the parity of $\text{dim}(V)$. If $\dim V$ is even, the Clifford algebra has a unique complex simple module $S$, with $\dim(S)=2^{\frac{\dim V}{2}}$. When $\dim V$ is odd, there are two different modules $S^{+}$ and $S^{-}$ of dimension $2^{\lfloor \frac{\dim V}{2} \rfloor}$. In the rest of the article, we will refer to any one of $S,S^+,S^-$ as a \textbf{spinor module}. In particular, $S$ is $\sgn$-dual to itself, whereas $S^{\pm} \otimes \sgn = S^{\mp}$. 

We now look into a distinguished subgroup of the Clifford algebra $C(V)$. Define an automorphism $\epsilon$ on $C(V)$ that is equal to $1$ on $C(V)^{[0]}$ and $-1$ on $C(V)^{[1]}$. Let $^t$ be the transpose anti-automorphism on $C(V)$ induced by $v^t=-v, v \in V$.

The \textbf{Pin group} is defined to be
$$ \Pin(V):=\{a \in C(V)^*: \epsilon(a)Va^{-1} \subseteq V, \; a^t=a^{-1} \} $$

The Pin group sits in a short exact sequence 
$$
\{1\} \rightarrow \Z/2\Z \rightarrow \Pin(V) \xrightarrow[]{p} \o(V) \rightarrow \{1\}
$$
where $p$ is defined to be $p(a).v=\epsilon(a)va^{-1}$.

Since $\Pin(V)$ generates $C(V)$, we know any spinor module $S$ is also a simple $\Pin(V)$-module. In addition, there's a well-known formula (see section II.6 of \cite{ctscohom}) 

$$ S \otimes S=\bigwedge\nolimits^\bullet V, \; \text{when dim V is even} $$
and
$$ S \otimes S=\bigoplus_{i=0}^{ \lfloor \dim V/2 \rfloor} \bigwedge\nolimits^{2i} V, \; \text{when dim V is odd} $$
as $\Pin(V)$-modules.

\begin{remark}[(5.5.5) of \cite{ciubotaru2012spin}]
\label{SFotimesSF}
 Let $\SF=S^{+}+S^{-}$ if $dim(V)$ is odd, or $\SF=S$ if $dim(V)$ is even. Then $$
 \SF \otimes \SF = \Big( \bigwedge\nolimits^\bullet V \Big)^{ \oplus a_{V}}  
 $$ as $\o(V)$-representations, where $a_{V} = 2$ in the odd case, and $a_{V} = 1$ in the even case.
\end{remark}

\subsection{Finite Coxeter Groups and Pin Covers}

This section is a quick reminder of several important constructions appearing in \cite{barbasch2012dirac}. We fix a finite Coxeter group $W$ associated with root system $(V_0,R)$, where $V_0$ is a finite dimensional real vector space. Choose a set of positive roots $R^+$, and let $\Pi\subseteq R^+$ be the corresponding set of simple roots. $W$ thus admits the following canonical presentation:
$$ W=\langle s_{\alpha}, \alpha \in \Pi: s_{\alpha}^2=1, (s_{\alpha}s_{\beta})^{m(\alpha,\beta)}=1, \; \text{for} \; \alpha, \beta \in \Pi \rangle, $$
where $m(\alpha,\beta) \in \mathbb{N}$ and $m(\alpha, \alpha) = 1$.

Let $V=V_0\otimes_\R \C$ be the complexified vector space. We may then view $W$ as a finite subgroup of $\o(V)$. From the degree 2 covering map between Lie groups $p:\Pin(V)\rightarrow \o(V)$, we can define a \say{double cover} of $W$:

$$ \Tw := p^{-1}(\o(V)) \subseteq \Pin(V) $$

This makes $\Tw$ a degree 2 central extension of $W$:
$$ \{1\} \rightarrow \Z/2\Z \rightarrow \Tw \rightarrow W \rightarrow \{1\} $$

One can see that $\Tw$ has the following presentation:
$$ \Tw=\big\langle z,f_\alpha, \alpha \in \Pi: z^2=1, f_\alpha^2=z, (f_\alpha f_\beta)^{m(\alpha,\beta)}=z, \; \text{for} \; \alpha \neq \beta \in \Pi \big\rangle. $$ $\Tw$ is called the \textbf{Pin cover} of the finite Coxeter group $W$.

This presentation is canonical in the sense that it's compatible with the embedding $\Tw \hookrightarrow \Pin(V)$, since 
there is a unique group homomorphism $\psi:\Tw \to \Pin(V)$ satisfying
$\psi(z)=-1$, and $\psi(f_{\alpha})=\alpha / |\alpha|$ for $\alpha \in \Pi$. Moreover, $\psi$ is injective, and we have $\psi(f_{\beta})=\beta / |\beta|$ for all $\beta \in R^+$.

A remarkable tool to study the representation theory of $\Tw$ is the following Casimir-type element introduced in \cite{barbasch2012dirac}:
$$\Omega_{\Tw}=z\sum_{\alpha, \beta>0, s_{\alpha}(\beta)<0}f_{\alpha}f_{\beta}.$$
It can be shown that $\Omega_{\Tw}$ is a central element of $\C[\Tw]$, see \cite{ciubotaru2012spin}.

\subsection{Representations of Weyl Groups and their Pin Covers}

Let $\mathfrak{g}$ be a complex semisimple Lie algebra with Cartan subalgebra $\mathfrak{h}$. Let $G$ be the simply connected Lie group with Lie algebra $\mathfrak{g}$. Let $W$ be the Weyl group of $\mathfrak{g}$, and thus we may view $W\subseteq \Pin(\mathfrak{h})$. Let $\mathcal{N}$ be the set of nilpotent elements of $\mathfrak{g}$. 

It is known from \cite{SpringerCorrespondence} that there is a parameterisation of irreducible representations of $W$ called the $\textbf{Springer correspondence}$, stated as follows.

Let $e \in \mathcal{N}$ be a nilpotent element. Let $Z_G(e)$ be the centraliser of $e$ in $G$ and $Z_G(e)^0$ be its identity component. Let $A(e)=Z_G(e)/Z_G(e)^0$ be its component group. Note that $A(e)$ actually only depends on the $G$-orbit of $e$. Let $\mathscr{B}^G$ be the variety of Borel subgroups of $G$. Let $\mathscr{B}^G_e$ be the closed subvariety of $\mathscr{B}^G$ consisting of Borel subgroups whose Lie algebra contains $e$. Let $d=\dim \mathscr{B}^G_e$.

In \cite{SpringerCorrespondence}, Springer showed that the top cohomology
$V_e=H^{2d}(\mathscr{B}^G_e,\C)$ admits two commuting actions of $A(e)$ and $W$. For $\phi \in \Irr(A(e))$, let $V_{e, \phi} := \text{Hom}_{A(e)}(\phi,V_e)$ be the isotypic component of $\phi$ in $V_{e}$ and let $\Irr_0(A(e)):=\{\phi\in \Irr(A(e)): V_{e,\phi} \neq 0\}$ be the set of such $\phi$ with non-zero isotypic components. 

The datum for the Springer correspondence is $\Sigma :=\{(e,\phi): e\in \mathcal{N}, \phi \in \Irr_0(A(e))\}$. Note that $G$ acts naturally on the set of nilpotent elements, and thus acts on $\Sigma$.

Finally, the Springer correspondence is the following bijection: $$\xi:G\backslash\Sigma \to \Irr_{}(W); \text{  } (e,\phi) \mapsto V_{e,\phi}$$

The representation theory of the Pin cover of $W$ is tied to that of $W$ in the following way. For every irreducible character $\chi$ of $\Tw$, the element $z$ acts by either $1$ or $-1$. If $z$ acts by $-1$, we say $\chi$ is a \textbf{genuine} irreducible representation of $\Tw$, for example the spinor module. Otherwise $\chi$ is non-genuine, and may be considered as a representation of $W$. Note, for example, that the tensor product of two genuine representations is non-genuine.

Let $\Irr_{\gen}(\Tw)$ be the set of genuine irreducible representations of $\Tw$. In \cite{ciubotaru2012spin}, the author constructed the following surjection:
$$ \Psi: \Irr_{\gen}(\Tw) \rightarrow G \backslash \mathcal{N}_{\sol}, $$
where $\mathcal{N}_{\sol}\subseteq\mathcal{N}$ is defined to be
$$ \mathcal{N}_{\sol}:=\{ e \in \mathcal{N}: \text{the centraliser of } e \text{ in } \g \text{ is a solvable Lie algebra} \} $$ The adjoint action of $G$ on $\mathfrak{g}$ gives rise to the sets $G \backslash \mathcal{N}_{\text{sol}} \subseteq G \backslash \mathcal{N}$.

\begin{definition}
    A nilpotent orbit $\O \in G \backslash \mathcal{N}$ is \textbf{solvable} if $\O \in G \backslash \mathcal{N}_{\text{sol}}$.
\end{definition}

By a version of Schur's lemma, we know the central element $\Omega_\Tw$ acts as scalar on each irreducible representation $\Tsig$ of $\Tw$. We denote this scalar as $\Tsig(\Omega_\Tw)$. Remarkably, this scalar has strong ties with the map $\Psi$.
\begin{theorem}[Theorem 1.0.1 of \cite{ciubotaru2012spin}, Theorem 5.1 of \cite{barbasch2012dirac}]
\label{generalisedSpringer}
The surjection $\Psi$ has the following nice properties:
\begin{enumerate}
    \item If $\Psi(\Tsig)=G\cdot e$, then $\Tsig(\Omega_\Tw)=\langle h_e,h_e\rangle$, where $h_e$ is the neutral element when extending $e$ into a Jacobson-Morozov triple.
    \item For a fixed spinor module $S$ of $C(V)$, if $e\in \mathcal{N}_{\text{sol}}$ and $\phi\in\Irr_0(A(e))$, then there exists $\Tsig\in \Psi^{-1}(G\cdot e)$ such that 
    \[
\Hom_W(\sigma_{e,\phi}, \Tsig\otimes S)\neq 0;
    \]
    \item If $\Psi(\Tsig)=G\cdot e$, then there exists $\phi\in\Irr_0(A(e))$ and a spinor module $S$ of $C(V)$ such that 
    \[
\Hom_W(\sigma_{e,\phi}, \Tsig\otimes S)\neq 0;
    \]
\end{enumerate}
\end{theorem}

The Springer correspondence and the representations of Pin covers of Weyl groups can both be phrased in terms of combinatorial data for classical types, which we will do so for the rest of this article whenever needed. More information on the Springer correspondence can be found in \cite{carter1993finite}.

\subsection{Parameterisations of \texorpdfstring{$\Irr(W), \ G \backslash \mathcal{N}, \text{ and } G \backslash \mathcal{N}_{sol} $}{}}

A full explanation of the parameterisation of $\Irr(W)$ may be found in section 3 of \cite{CIUBOTARU20221}, although we restate the essentials here for convenience. Let $P(n)$ be the partitions of $n \in \mathbb{N}$, $BP(n)$ the bi-partitions of $n$, and $DP(n)$ the strict partitions of $n$.

\begin{theorem}[Parameterisation of $\Irr(W)$ for Classical Types, \cite{CIUBOTARU20221}]

Let $\mathfrak{g}$ be a simple Lie algebra of classical type with Weyl group $W$. Then $\Irr(W)$ are parameterised by:

\renewcommand{\arraystretch}{1.6}
\begin{center}
\begin{tabular}{ |p{2cm}|p{6cm}|p{6cm}| }
\hline
$\mathfrak{g}$ & Parameterisation & Notes \\
\hline
$A_{n-1}$ & $\sigma = \sigma_{\lambda} $ for $\lambda \in P(n)$ & (via Young diagrams \cite{fulton1991representation}) \\
\hline
$B_{n}$ / $C_{n}$ & $\sigma = \sigma_{\lambda, \mu} $ for $(\lambda \times \mu) \in BP(n)$ & (by induction from type $A$) \\
\hline
$D_{n}$ & $\sigma = \sigma_{\lambda, \mu} = \sigma_{\mu, \lambda}$ for $\lambda \neq \mu$, & (by restriction from type B/C) \\
& otherwise $\sigma = \sigma_{\lambda, \lambda}^{+}$ or $\sigma = \sigma_{\lambda, \lambda}^{-}$ & \\
\hline
\end{tabular}
\end{center}
\renewcommand{\arraystretch}{1.0}
\end{theorem}

For the exceptional types, we follow Carter's notation \cite{carter1993finite}.

We now recall the parameterisations of $G \backslash \n$ and $G \backslash \n_{sol}$ for classical types. The second column is from Chapter 5 of \cite{collingwood1993nilpotent}, and the third column is from Section 3 of \cite{CIUBOTARU20221}.

\begin{theorem}[Parameterisations of $G \backslash \mathcal{N}$ and $G \backslash \mathcal{N}_{sol}$ for Classical Types, \cite{collingwood1993nilpotent}, \cite{CIUBOTARU20221}] 

Let $\g$ be a simple Lie algebra of classical type, with nilpotent orbits $G\backslash\mathcal{N}$ and solvable nilpotent orbits $G\backslash\mathcal{N}_{\text{sol}}$. Then we have the following parameterisation: 

\renewcommand{\arraystretch}{1.6}
\begin{center}
\begin{tabular}{ |p{1cm}|p{7cm}|p{7cm}| }
\hline
$\mathfrak{g}$ & $G \backslash \mathcal{N}$ & $G \backslash \mathcal{N}_{sol}$ \\
\hline
$A_{n-1}$ & $P(n)$ & $DP(n)$ \\
\hline
$B_{n}$ & $P_1(2n+1) \subseteq P(2n+1)$, partitions whose even parts have even multiplicity & Elements of $P_1(2n+1)$ where all parts are odd, and distinct parts have multiplicity $\leq 2$ \\
\hline
$C_{n}$ & $P_{-1}(2n) \subseteq P(2n)$, partitions whose odd parts have even multiplicity & Elements of $P_{-1}(2n)$ where all parts are even, and distinct parts have multiplicity $\leq 2$ \\
\hline
$D_{n}$ & $P_1(2n) \subseteq P(2n)$, partitions whose even parts have even multiplicity - except that the partitions with all even parts parameterise two distinct orbits & Elements of $P_1(2n)$ where all parts are odd, and distinct parts have multiplicity $\leq 2$ \\
\hline
\end{tabular}
\end{center}
\renewcommand{\arraystretch}{1.0}

\end{theorem}

For future use, we will need a particular partial order on $G \backslash \n$: the closure ordering.

\begin{definition}
    Let $\g$ be a Lie algebra of classical types, and $\lambda, \mu$ be two partitions parameterising nilpotent orbits $\O_{\lambda}, \O_{\mu} \in G \backslash \n$. We define a relation $\leq$ on $G \backslash \n$, where $\O_{\lambda} \leq \O_{\mu}$ if and only if $\overline{\O_{\lambda}} \subseteq \overline{\O_{\mu}}$. This relation is a partial order, called the \textbf{closure ordering} on $G \backslash \n$.
\end{definition}

The closure ordering can be described combinatorially in classical types. In fact, it almost coincides with the dominance order on associated partitions.

\begin{definition}
    There exists a partial order on the set of partitions of $n$, called the \textbf{dominance order}, $\leq_{\text{dom}}$. For $\lambda, \mu \in P(n)$, set $\lambda=(x_1,...,x_n)$ and $\mu=(y_1,...,y_n)$, with the parts in non-increasing order, and allowing zeroes. We say $\lambda \leq_{\text{dom}} \mu$ if and only if $\sum_{i=1}^{k} x_i \leq \sum_{i=1}^{k} y_i$ for any $1 \leq k \leq n$.
\end{definition}

\begin{theorem}[Theorem 6.2.5 of \cite{collingwood1993nilpotent}]
    Let $\g$ be simple Lie algebra of classical type, and $\O_\lambda,\O_\mu$ be nilpotent orbits attached to partitions $\lambda,\mu$. Then
    $\O_{\lambda} < \O_{\mu}$ if and only if $\lambda <_{\text{dom}} \mu$.
\end{theorem}

\begin{remark}
In fact we can replace $<$ by $\leq$ in the above theorem, except for type D. Extra care must be taken in type D since there can be 2 orbits attached to partition that are incomparable.
\end{remark}

For $\tilde{\sigma} \in \Irr_{\text{gen}}(\Tw)$, the ordering of scalars $\tilde{\sigma}(\Omega_\Tw)$ is compatible through $\Psi$ with the closure ordering on the corresponding solvable  orbits.

\begin{remark}
\label{lengthofh}
     Let $G \cdot e_1, G \cdot e_2$ be two nilpotent orbits such that $e_1 \geq e_2$ in the closure ordering. Then 
    $ \langle h_{e_1}, h_{e_1} \rangle \geq \langle h_{e_2},h_{e_2} \rangle. $ See Corollary 3.3 of \cite{gunnells2002characterization} for a uniform proof.

    In particular, we have the following as an easy corollary from Theorem \ref{generalisedSpringer} (1). Let $\tilde{\sigma}_{1}, \tilde{\sigma}_{2} \in \Irr_{\text{gen}}(W)$ such that $\Psi(\tilde{\sigma}_{1}) \geq \Psi(\tilde{\sigma}_{2})$ in the closure ordering. Then 
    $ \tilde{\sigma}_{1}(\Omega_\Tw) \geq \tilde{\sigma}_{2}(\Omega_\Tw) $.
\end{remark}

\subsection{The Spaltenstein Map}

In this section, we give a brief review of the Spaltenstein map between nilpotent orbits, which is a natural generalisation of the transpose map between partitions for type A. We first give the formal definition of Spaltenstein map, which is essentially the same as Spaltenstein's original definition in \cite{Spaltensteindmapbook}.

\begin{definition}
Fix a semisimple Lie algebra $\g$, its connected reductive Lie group $G$, and $G \backslash \n$ the ordered (natural ordering based on closure relations) set of nilpotent orbits of $\g$. Let $\p$ be a parabolic subalgebra of $\g$. Let $\O_\p=\Ind_{\p}^{\g}(\O_0) \in G \backslash \n$ be the `Richardson orbit' corresponding to $\p$. We also attach to $\p$ a nilpotent orbit $\hat{\O_\p} \in G \backslash \n$ containing the regular nilpotent elements of a Levi subalgebra $\mathfrak{l}$ of $\p$. The \textbf{Spaltenstein map} $d:G \backslash \n \to G \backslash \n$ is a map satisfying the following conditions: \begin{enumerate}
\item For all parabolic subalgebras $\p$ of $\g$, $d(\hat{\O_\p})=\O_\p$.
\item $x \leq d^2(x)$ for all $x \in G \backslash \n$.
\end{enumerate}
\end{definition}

\begin{remark}
The above properties uniquely determine the map for classical types and $E_7$. For the remaining exceptional types, a choice can be made, see Section 9 of \cite{Spaltensteindmapbook} for details.
\end{remark}

In type A, the map sending the nilpotent orbit $\mathcal{O}_{\lambda}$ to $\mathcal{O}_{\lambda^t}$ satisfies the conditions above, because we can see that $\hat{\O}_{\p(d)}=\O_d$, and $\O_{\p(d)}=\O_{d^t}$. The uniqueness of this map is confirmed by Theorem 1.5, Chapter 3 of \cite{Spaltensteindmapbook}.

The explicit description of Spaltenstein's map in other classical types relies on the notion of B (resp. C,D)-collapses.

\begin{definition}[Lemma 6.3.3 of \cite{collingwood1993nilpotent}]
\label{CollapseOfPartition}
    Let $\lambda=(a_1,...,a_{2n+1}) \in P(2n+1)$ be any partition of $2n+1$, allowing zeroes. There exists a unique largest partition in $P_1(2n+1)$ dominated by $\lambda$, called $\lambda_B$, the \textbf{B-collapse} of $\lambda$. We obtain $\lambda_B$ from $\lambda$ in the following way. If $\lambda$ is not in $P_1(2n+1)$, then it must have even parts occurring with odd multiplicity. Let the largest such part be $q$. Replace the last occurrence of $q$ by $q-1$, and replace the first subsequent part $r$ strictly less than $q-1$ by $r+1$. Note it may be the case that $r=0$. Repeat this process until we obtain a partition in $P_1(2n+1)$. We obtain the \textbf{C-collapse and D-collapse} of some $\lambda \in P(2n)$ is the same way. They are also unique largest partitions in $P_{-1}(2n)$ and $P_1(2n)$ respectively that are dominated by $\lambda \in P(2n)$.
\end{definition}

\begin{proposition}[Theorem 6.3.5 in \cite{collingwood1993nilpotent}]
\label{SpaltensteinTypeBCD}
The maps $d:\lambda \mapsto (\lambda^t)_{B}$ (resp. $\lambda \mapsto (\lambda^t)_{C}, \lambda \mapsto (\lambda^t)_{D}$ are the unique Spaltenstein maps for type B,C, and D.
\end{proposition}

\begin{definition}
    We call a nilpotent orbit $\O$ \textit{special} if it lies in the range of the Spaltenstein map $d$. We also call a character $\sigma \in\Irr(W)$ special if $\sigma =\sigma_{\O,\triv}$ for some special orbit $\O$.
\end{definition}

To conclude this section, we summarise the properties of Spaltenstein's map, special orbits, and special characters for future use:
\begin{enumerate}
    \item $d^2(\O)\ge\O$.
    \item $d$ is a order-reversing map: $\O_1\ge \O_2\implies d(\O_1)\le d(O_2)$.
    \item $d$ is an involution on the special orbits.
    \item $d$ corresponds to tensoring with $\sgn$ on the level of special characters via the Springer correspondence. Namely, if $\O$ is a special orbit, then $\sigma_{\O,\triv}\otimes\sgn=\sigma_{d(\O),\triv}$.

\end{enumerate}

\section{Lusztig Families and Restriction on \texorpdfstring{$W$}{}-types}

Let $\mathfrak{g}$ be a simple Lie algebra with Cartan subalgebra $\mathfrak{h}$ and Weyl group $W$. Fix a choice of the unique spinor modules $S^{+}$ and $S^{-}=S^{+}\otimes \sgn$ (resp. $S$) of $\Pin(\mathfrak{h})$, if dim$(\mathfrak{h})$ is even (resp. odd). Set $\SF=S^{+}+S^{-}$ (resp. $\SF=S$) if dim($\mathfrak{h})$ is even (resp. odd).

Our approach uses the following theorem on restricting the $\widetilde{W}$-types that can occur in $\sigma \otimes \SF$, for fixed $\sigma \in \Irr(W)$. 

\begin{theorem}[Corollary 6.10 of \cite{CIUBOTARU20151}]
\label{controlwtype}
Let $e\in \mathcal{N}$, $\phi\in \Irr_0(A(e))$, and $\widetilde{\sigma} \in \Irrg(\Tw)$. If $\langle\sigma_{e,\phi}\otimes \SF,\widetilde{\sigma}\rangle_\Tw\neq 0$, then $\widetilde{\sigma}$ corresponds via $\Psi$ to some $e'\in \mathcal{N}_{\text{sol}}$ for $e'\ge e$ in the closure ordering.
\end{theorem}

\begin{remark}
Since $\SF=\SF^*$, we always have $\langle\sigma_{e,\phi}\otimes \SF,\widetilde{\sigma}\rangle_\Tw=\langle\sigma_{e,\phi},\widetilde{\sigma}\otimes \SF\rangle_W$ by duality. So the theorem can also interpreted as restricting the $W$-types that can occur in $\widetilde{\sigma}\otimes\SF$ for a fixed $\widetilde{\sigma}\in \Irrg(\Tw)$.
\end{remark}

The above theorem only gives us an upper bound for $W$-types appearing in $\Tsig\otimes\SF$. However, a lower bound can also be achieved using Spaltenstein's d-map and the so-called Lusztig families, which gives us control over the possible $W$-types that can appear in the tensor product.

We now briefly describe Lusztig's construction of a partition of $\Irr(W)$ into \textbf{Lusztig families} as in \cite{lusztig1982class} (see also \cite{carter1993finite}). First, Lusztig introduced a notion of $\Tilde{j}$-induction, which is roughly speaking a truncated induced character from a parabolic subgroup. One can then define the notion of \say{cells} inductively. We can assign an integer called the \textbf{a-value} to each cell.

\begin{definition}
    The set of cells is the smallest set of (not necessarily irreducible) representations of $W$ satisfying the following conditions:
    \begin{enumerate}
        \item for $W=\{1\}$, the unit representation is a cell.
        \item if $\phi$ is a cell of a parabolic subgroup $W_J$, then $\Tilde{j}(\phi)$ and $\Tilde{j}(\phi)\otimes\sgn$ are cells of $W$.
    \end{enumerate}
\end{definition}

The cells have the property that each irreducible character of $W$ is a component of at least one cell; each cell contains a unique special character as component with multiplicity 1; two cells for which the special components are distinct have no common irreducible components.
We say two characters $\phi,\phi'$ are equivalent if there exists cells $c,c'$ such that:
\begin{enumerate}
    \item $\phi$ appears in $c$, $\phi'$ appears in $c'$.
    \item the unique special characters in $c$ and $c'$ coincide.
\end{enumerate}
This equivalence relation induces the partition of $\Irr(W)$ into Lusztig families. For the classical types, there are combinatorial descriptions of the Lusztig families using so-called \say{symbols}. We will refer to this description whenever needed.

From the construction above, one easily sees that each family $F$ contains a special character $\sigma_F$. Via the Springer correspondence, we can associate a special orbit $\mathcal{O}_F$ with $\sigma_F$ for each family $F$. Moreover, the notion of family interacts well with the Spaltenstein map as follows. Let $\sigma$ be in Lusztig family $F$, and $\O_\sigma$ the associated nilpotent orbit via the Springer correspondence. Then $d^2(\O_\sigma)$ is the unique minimal special orbit above $\O_\sigma$, which is precisely $\O_F$. In addition, $\sigma\otimes\sgn$ lies in the Lusztig family $F'$ such that $\O_{F'}=d(\O_F)$.

An important observation is the following. According to our definition of $\SF$, we see that $\SF \otimes \sgn=\SF$. Consider $\sigma \in F$. We have the identity $\sigma \otimes \SF =\sigma \otimes \sgn \otimes \SF$, and we may view $\sigma\otimes \sgn$ as a $W$-type in Lusztig family $F'$ such that $\mathcal{O}_{F'}=d(\mathcal{O}_F)$. Hence the next corollary follows:

\begin{corollary}
\label{enhancedwtype}
Let $\widetilde{\sigma}\in \Irrg(\Tw)$ be attached to a solvable nilpotent orbit $\mathcal{O}$. Suppose $\langle\sigma ,\widetilde{\sigma}\otimes \SF\rangle_W\neq0$ and $\sigma$ belongs to Lusztig family $F$. Then $\mathcal{O} \ge \mathcal{O}_F \ge d(\mathcal{O})$.
\end{corollary}

\begin{proof}
Let $\O_\sigma$ be the nilpotent orbit corresponding to $\sigma$ via Springer correspondence. By Theorem \ref{controlwtype}, we know $\mathcal{O}_\sigma \leq \O$. Since $d^2$ preserves ordering, we see that $\O_F=d^2(\O_\sigma)\leq d^2(\O)=\O$. Similarly, by the above discussion we know $\O_{F'}\leq \O$. It follows that $\O_F=d(\O_{F'}) \ge d(\O)$, since $d$ is an order reversing involution on the set of special orbits.
\end{proof}

\section{Saxl Conjecture and its Generalisation}

\subsection{Motivation and the Generalised Saxl Conjecture}

Corollary \ref{enhancedwtype} is particularly powerful when $\widetilde{\sigma}$ corresponds to a \say{small} solvable orbit $\O$ with respect to the closure ordering. As shown in \cite{CIUBOTARU20221}, there exists a unique minimal solvable orbit for each simple Lie algebra. Therefore, it's natural to wonder about its consequences regarding the minimal solvable orbit of a simple Lie algebra.

As a first example, Corollary \ref{enhancedwtype} can already be used to efficiently attack Saxl's conjecture, which we restate here.

\begin{conjecture}[Saxl]
    Let $\lambda=(k,k-1,\cdots,1)$ be the staircase partition of $n=k(k+1)/2$. Then $\sigma_\lambda \otimes \sigma_\lambda$ contains all irreducible characters of $S_n$ as constituents.
\end{conjecture}

We take the viewpoint that $S_n$ is the Weyl group of Lie algebra $\mathfrak{g}=\mathfrak{sl}_{n}$. The representation theory of $S_n$ and its Pin cover is phrased in terms of partitions of $n$, and a nilpotent orbit is solvable if and only if it corresponds to a strict partition. Each $S_{n}$-type is its own Lusztig family, and the map $d$ is given by transpose $^t$ of the Young diagrams.

The observation is that in the setup of Saxl's conjecture, $\lambda$ corresponds to the minimal solvable orbit $\O$ of $\mathfrak{g}$. Additionally, $\lambda=\lambda^t$ and therefore $\O=d(\O)$. Corollary \ref{enhancedwtype} implies that any $W$-types appearing in $\widetilde{\sigma}\otimes \SF$ must lie in the Lusztig family $F$ corresponding to $\O$, which consists of a single element $\lambda$. The following result was proved by Bessenrodt in \cite{BessenrodtSaxlConjecture} by considering spin representations. However, we still provide a short proof that fits well into our setting.

\begin{proposition}
    \label{saxlA-1}
    Let $\lambda=(k,k-1,\cdots,1)$ be the staircase partition of $n=k(k+1)/2$. Then $\sigma_\lambda \otimes \sigma_\lambda$ contains all irreducible characters of $S_n$ corresponding to hooks, i.e. partitions of the form $(n-m,1^m)$.
\end{proposition}

\begin{proof}
Let $\widetilde{\sigma}_\lambda$ be the unique genuine representation of $\widetilde{S_n}$ attached to $\lambda$ by the map $\Psi$. By Corollary \ref{enhancedwtype}, we see the only $S_{n}$-type inside $\widetilde{\sigma}_\lambda \otimes \SF$ is $\sigma_\lambda$. Therefore ${\sigma_\lambda}^{\otimes 2}$ and $(\widetilde{\sigma}_\lambda\otimes \SF)^{\otimes 2}$ have the same constituents. But $(\widetilde{\sigma}_\lambda\otimes \SF)^{\otimes 2}=(\widetilde{\sigma}_\lambda\otimes \widetilde{\sigma}_\lambda)\otimes (\SF \otimes \SF)=a_V(\widetilde{\sigma}_\lambda\otimes \widetilde{\sigma}_\lambda)\otimes \bigwedge^{\bullet}(V)$ using Remark \ref{SFotimesSF}.

Since $\sigma_\lambda$ has real character, so does $(\widetilde{\sigma}_\lambda \otimes \SF)$. But $\SF$ also has real character, so $\widetilde{\sigma}_\lambda$ does also. Hence $\widetilde{\sigma}_\lambda = \widetilde{\sigma}_\lambda^{*}$ is its own contragredient, so $\widetilde{\sigma}_\lambda\otimes \widetilde{\sigma}_\lambda = \text{Hom}(\widetilde{\sigma}_\lambda, \widetilde{\sigma}_\lambda)$ contains a copy of the trivial representation. 

On the other hand, each graded component of $\bigwedge^{\bullet}(V)$ is an irreducible $S_{n}$-type:
$\bigwedge^m(V)$ corresponds to the irreducible character indexed by the hook partition $(n-m,1^m)$.
\end{proof}

With this setup, we can naturally generalise Saxl's conjecture to Weyl groups. 

\begin{conjecture}[Generalised Saxl Conjecture for Weyl Groups]
\label{gen_Saxl_conj_ver_I}
    Let $\mathfrak{g}$ be a simple Lie algebra with Weyl group $W$. Let $\O$ be the unique minimal solvable orbit of $\mathfrak{g}$.
    Then $$\Big(\bigoplus_{\sigma\in F: \ \O \ge \O_F \ge d(\O)} \sigma \ \Big)^{\otimes 2} $$ contains all irreducible characters of $W$, where $\O_F$ is the special orbit attached to family $F$.
\end{conjecture}

Recall that each exterior power of the reflection representation is an irreducible character of the Weyl group. One can thus easily generalise the argument in Proposition \ref{saxlA-1} to obtain the following:

\begin{proposition}
    Let $\mathfrak{g}$ be a simple Lie algebra with Weyl group $W$. Let $\O$ be the unique minimal solvable orbit of $\mathfrak{g}$.
    Then $(\bigoplus_{\sigma\in F:\O \ge \O_F \ge d(\O)} \sigma)^{\otimes 2} $ contains all irreducible characters of $W$ of the form $\bigwedge^m(V)$, where $V$ represents the reflection representation of $W$.
\end{proposition}

\begin{proof}
Pick any $\widetilde{\sigma}\in \Irrg(\Tw)$ corresponding to orbit $\O$, and thus the constituents of $\widetilde{\sigma}\otimes \SF$ all lie in the summand $\bigoplus_{\sigma\in F:\O \ge \O_F \ge d(\O)} \sigma$. Observe that all characters of $W$ are integral (in particular real), and thus $\widetilde{\sigma}$ has real character as before. We conclude the proof by imitating the arguments in the proof of Proposition \ref{saxlA-1}.
\end{proof}

\begin{example}
    In general, the summand $\bigoplus_{\sigma\in F:\O \ge \O_F \ge d(\O)} \sigma$ may be quite large. For example, consider $W=S_9$. The minimal solvable orbit is given by $\lambda=(432)$. There are 5 partitions between $\lambda=(432)$ and $\lambda^t=(3321)$, namely $(432),(4311),(4221),(333),(3321)$. One can easily verify that $((432)+(3321))^{\otimes 2}$ already contains all irreducible characters of $S_9$, and therefore the generalised Saxl conjecture is quite weak in this case.
\end{example}

However, there is a particularly interesting case when the minimal solvable orbit is \textbf{self-dual}: $\O=d(\O)$. In this case, we just sum over all characters in family $F$ corresponding to $\O$. Remarkably, we have a large supply of minimal solvable orbits that are self-dual: 

\begin{enumerate}
    \item $n=k(k+1)/2$ for type A, $\O=(k,k-1,\cdots,1)$.
    \item $n=k(k+1)$ for type C, $\O=(2k,2k,\cdots,4,4,2,2)$.
    \item $n=k^2$ for type D, $\O=(2k-1,2k-1,\cdots,3,3,1,1)$.
    \item Every exceptional type: $G_2(a_1)$ for $G_2$; $F_4(a_3)$ for $F_4$; $D_4(a_1)$ for $E_6$; $A_4+A_1$ for $E_7$; $E_8(a_7)$ for $E_8$.
\end{enumerate}

We thus highlight the following special case of the generalised Saxl conjecture, which is particularly interesting as the direct sum is \say{small}:

\begin{conjecture}[Generalised Saxl Conjecture for Weyl Groups, Ver. II]
\label{gen_Saxl_conj_ver_II}
Let $\mathfrak{g}$ be a simple Lie algebra with Weyl group $W$. Let $\O_F$ be the unique minimal solvable orbit of $\mathfrak{g}$, where $\O_{F}$ is attached to Lusztig family $F$. Suppose also that $\O_F=d(\O_F)$.
    Then $(\bigoplus_{\sigma\in F} \sigma)^{\otimes 2} $ contains all irreducible characters of $W$.
\end{conjecture}

\begin{remark}
\label{cuspidal}
The notion of (1-stable) cuspidal family was introduced by Lusztig in Chapter 8 of \cite{lusztig1984characters} via $\Tilde{j}$-induction from parabolic subgroups. It can be shown that there is at most one cuspidal family for each Weyl group. The following remarkable fact can be seen by a case-by-case analysis. Except for type $A$, a cuspidal family $\widetilde{F}$ exists if and only if there is a self-dual solvable orbit $\O$. Suppose that $\O = \O_{F}$ for Lusztig family $F$. Then in fact $\widetilde{F} = F$, i.e. the cuspidal family is equal to the Lusztig family attached to the unique solvable self-dual orbit. 
\end{remark}

We first prove Conjecture \ref{gen_Saxl_conj_ver_II} for the exceptional types by direct calculation.

\begin{proposition}
Conjecture \ref{gen_Saxl_conj_ver_II} holds for the Weyl groups of exceptional types.
\end{proposition}

\begin{proof}
By Proposition 5.5 of \cite{CIUBOTARU20221}, the minimal solvable orbits for type $G_2,F_4,E_6,E_7,E_8$ are those corresponding to unipotent classes $G_2(a_1),F_4(a_3),D_4(a_1),A_4+A_1,E_8(a_7)$ respectively, and they're all self-dual. The statement can be verified case-by-case using GAP 3 \cite{GAP3} and package CHEVIE \cite{CHEVIE}. We demonstrate this using the particularly nice example $E_7$, for which the family $F$ of the minimal solvable orbit only has two characters, $F = \{ \phi_{512,11},\phi_{512,12} \}$.\\ 

Let $\phi\in\Irr(W(E_7))$. Suppose $\phi$ occurs with multiplicity $m$ in $(\phi_{512,11}+\phi_{512,12})^{\otimes2}$, then we list the data as: $\phi \ (m)$. The decomposition is as follows:

\renewcommand{\arraystretch}{1.2}
\begin{center}
\begin{tabular}{ l l l l l l }
$\phi_{1,0}\ (2)  $
& $\phi_{1,63}\ (2)  $
& $\phi_{7,46}\ (4)  $
& $\phi_{7,1}\ (4)  $
& $\phi_{15,28}\ (6)  $
& $\phi_{15,7}\ (6)  $
\\
$\phi_{21,6}\ (10)  $
& $\phi_{21,33}\ (10)  $
& $\phi_{21,36}\ (10)  $
& $\phi_{21,3}\ (10)  $
& $\phi_{27,2}\ (12)  $
& $\phi_{27,37}\ (12)  $
\\
$\phi_{35,22}\ (14)  $
& $\phi_{35,13}\ (14)  $
& $\phi_{35,4}\ (14)  $
& $\phi_{35,31}\ (14)  $
& $\phi_{56,30}\ (24)  $
& $\phi_{56,3}\ (24)  $
\\
$\phi_{70,18}\ (26)  $
& $\phi_{70,9}\ (26)  $
& $\phi_{84,12}\ (30)  $
& $\phi_{84,15}\ (30)  $
& $\phi_{105,26}\ (40)  $
& $\phi_{105,5}\ (40)  $
\\
$\phi_{105,6}\ (40)  $
& $\phi_{105,21}\ (40)  $
& $\phi_{105,12}\ (40)  $
& $\phi_{105,15}\ (40)  $
& $\phi_{120,4}\ (46)  $
& $\phi_{120,25}\ (46)  $
\\
$\phi_{168,6}\ (62)  $
& $\phi_{168,21}\ (62)  $
& $\phi_{189,10}\ (70)  $
& $\phi_{189,17}\ (70)  $
& $\phi_{189,22}\ (70)  $
& $\phi_{189,5}\ (70)  $
\\
$\phi_{189,20}\ (70)  $
& $\phi_{189,7}\ (70)  $
& $\phi_{210,6}\ (80)  $
& $\phi_{210,21}\ (80)  $
& $\phi_{210,10}\ (72)  $
& $\phi_{210,13}\ (72)  $
\\
$\phi_{216,16}\ (76)  $
& $\phi_{216,9}\ (76)  $
& $\phi_{280,18}\ (106)  $
& $\phi_{280,9}\ (106)  $
& $\phi_{280,8}\ (98)  $
& $\phi_{280,17}\ (98)  $
\\
$\phi_{315,16}\ (112)  $
& $\phi_{315,7}\ (112)  $
& $\phi_{336,14}\ (122)  $
& $\phi_{336,11}\ (122)  $
& $\phi_{378,14}\ (134)  $
& $\phi_{378,9}\ (134)  $
\\
$\phi_{405,8}\ (146)  $
& $\phi_{405,15}\ (146)  $
& $\phi_{420,10}\ (152)  $
& $\phi_{420,13}\ (152)  $
& $\phi_{512,12}\ (182) $
& $\phi_{512,11}\ (182)  $
\end{tabular}
\end{center}
The tables bearing the proof for the other exceptional types can be found in the appendix.
\end{proof}

\begin{remark}
\label{fail_over_min_solv_orb}
One might wonder whether it suffices to sum just over the characters attached to the minimal solvable orbit (via the Springer correspondence), instead of everything within the corresponding Lusztig family, before tensoring with itself to get all irreducible characters. However, this fails for all exceptional types except $E_7$ (in which case, every character in the Lusztig family corresponds to the unipotent class $A_4+A_1$.)
\end{remark}

\begin{example} We simultaneously show the verification of Conjecture \ref{gen_Saxl_conj_ver_II} and the failure of Remark \ref{fail_over_min_solv_orb} for $\mathfrak{g} = F_{4}$. The self-dual minimal solvable orbit is $F_{4}(a_{3})$. We use Carter's notation \cite{carter1993finite} for the characters. The corresponding family $F$ (which has 11 out of 25 characters) and $F_{4}(a_{3})$ are: $$F = \{
\phi_{12,4},\
\phi_{16,5}, \
\phi_{6,6}', \
\phi_{6,6}'', \
\phi_{9,6}', \
\phi_{9,6}'', \
\phi_{4,7}', \
\phi_{4,7}'', \ 
\phi_{4,8},  \
\phi_{1,12}', \ 
\phi_{1,12}''
\}$$
$$F_{4}(a_{3}) = \{
\phi_{12,4},\
\phi_{6,6}'', \
\phi_{9,6}', \
\phi_{1,12}'
\} $$

Let $\phi \in \Irr(W(F_{4}))$. Suppose $\phi$ occurs with:

\begin{enumerate}
    \item multiplicity $m_{1}$ in $(\bigoplus_{\sigma\in F} \sigma)^{\otimes 2}$ as in Conjecture \ref{gen_Saxl_conj_ver_II}
    
    \item multiplicity $m_{2}$ in $(\bigoplus_{\sigma\in F_{4}(a_{3})} \sigma)^{\otimes 2}$ as in Remark \ref{fail_over_min_solv_orb}
    
\end{enumerate}

Then we list the data as: $\phi \; (m_{1})\; (m_{2})$. The decomposition of each is as follows:

\renewcommand{\arraystretch}{1.2}
\begin{center}
\begin{tabular}{ l l l l l }
$\phi_{1,0}\ (11)\ (4)$
& $\phi_{1,12}''\ (5)\ (1)  $
& $\phi_{1,12}'\ (5)\ (1)  $
& $\phi_{1,24}\ (11)\ (2)  $
& $\phi_{2,4}''\ (9)\ (4)  $
\\
$\phi_{2,16}'\ (9)\ (3)  $
& $\phi_{2,4}'\ (9)\ (4)  $
& $\phi_{2,16}''\ (9)\ (3)  $
& $\phi_{4,8}\ (25)\ (7)  $
& $\phi_{9,2}\ (46)\ (15)  $
\\
$\phi_{9,6}''\ (40)\ (11)  $
& $\phi_{9,6}'\ (40)\ (11)  $
& $\phi_{9,10}\ (46)\ (11)  $
& $\phi_{6,6}'\ (35)\ (10)  $
& $\phi_{6,6}''\ (39)\ (6)  $
\\
$\phi_{12,4}\ (57)\ (16)  $
& $\phi_{4,1}\ (24)\ (0)  $
& $\phi_{4,7}''\ (16)\ (0)  $
& $\phi_{4,7}'\ (16)\ (0)  $
& $\phi_{4,13}\ (24)\ (0)  $
\\
$\phi_{8,3}''\ (28)\ (0)  $
& $\phi_{8,9}'\ (28)\ (0)  $
& $\phi_{8,3}'\ (28)\ (0)  $
& $\phi_{8,9}''\ (28)\ (0)  $
& $\phi_{16,5}\ (68)\ (0)  $

\end{tabular}
\end{center}
 
\end{example}

\subsection{Generalised Saxl Conjecture for Type C}

In this section, we look into Saxl's conjecture for type C. Let $W=W(C_n)=(\Z/2\Z)^n \rtimes S_n$ be the Weyl group for $\mathfrak{sp}(2n)$. Recall that the Springer correspondence in this setup can be phrased by the following combinatorial data (e.g. Chapter 13 of \cite{carter1993finite}):

\[\begin{tikzcd}
	{\text{Irr}(W)} &&& {\mathcal{N}} \\
	\\
	{\text{BP}(n)} &&& {P_{-1}(2n)}
	\arrow[two heads, from=1-1, to=1-4]
	\arrow["\sim"', from=1-1, to=3-1]
	\arrow["\sim", from=1-4, to=3-4]
	\arrow[two heads, from=3-1, to=3-4]
	\arrow["\psi"', shift left, bend right, from=3-4, to=3-1],
\end{tikzcd}\]
Let BP($n$) be the set of bi-partitions of $n$, and $P_{-1}(2n)$ be the partitions of $2n$ in which odd parts appear an even number of times. The section map $\psi$ gives us a recipe for constructing a bi-partition from an element in $P_{-1}(2n)$, such that the image of $\psi$ precisely corresponds to the characters of $W$ of the form $\sigma_{e,\triv}$ for $\triv \in \Irr_{0}(A(e))$ the trivial character in the component group of $Z_G(e)$. 

A nilpotent orbit is solvable if and only if the corresponding partition consists of only even parts, and all parts have multiplicity at most two. The Spaltenstein map $d:P_{-1}(2n) \rightarrow P_{-1}(2n)$ is given by $\nu \mapsto \nu^t_C$, where the subscript $C$ represents the $C$-collapse. It is clear from the definition that an orbit $\nu\in P_{-1}(2n)$ is special if and only if $\nu^t\in P_{-1}(2n)$.

We now state the following explicit description of the genuine representations of $\widetilde{W}$:

\begin{theorem} [Theorem 5.1 of \cite{ReadProjectiveReps}]
    \label{Tw_parameterisation}
    Let $\Tw$ be the Pin cover of the Weyl group of type $C_n$. There is a one-to-one correspondence
    \[
    \Irrg(\Tw) / \sim  \quad \leftrightarrow \quad P(n)
    \]
    where $\sim$ is the equivalence relation generated by $\widetilde{\sigma}\sim \widetilde{\sigma}\otimes \sgn$.

For each $\lambda \in P(n)$, there is exactly:
\begin{enumerate}
    \item one irreducible $\Tw$-type $\sigma_{(\lambda,\emptyset)} \otimes S$, if $n$ is even;
    \item two associate irreducible $\Tw$-types $\sigma_{(\lambda,\emptyset)} \otimes S^{\pm}$, if $n$ is odd.
\end{enumerate}
\end{theorem}

Let $\mu$ be the unique minimal solvable orbit constructed in Proposition 5.5 of \cite{CIUBOTARU20221}, and let $\lambda$ be a partition of $n$ corresponding to the genuine representation associated with $\mu$ in the above setting. For a special orbit $\nu \in P_{-1}(2n)$, denote $F_\nu$ as its corresponding Lusztig family. Combined with the the characterisation of closure ordering in Theorem 6.2.5 of \cite{collingwood1993nilpotent}, the above discussion leads us to the following nice consequence:

\begin{corollary}
\label{exterior_power_tensor_appears_in_type_C}
    Let $W$ be the Weyl group of $\mathfrak{sp}(2n)$, and $\mu$ be the minimal solvable orbit. Then all constituents of $\bigwedge^i(V) \otimes \bigwedge^j(V)$ also appear in $$\Big(\bigoplus_{\sigma\in F_\nu: \ \mu \ge \nu \ge \mu^t_C} \sigma \ \Big)^{\otimes 2} $$ where $V$ represents the reflection representation of $W$, and $0\leq j\leq i\leq n$.
\end{corollary}

\begin{proof}
We know $\sigma_{(\lambda,\emptyset)} \otimes \SF$ corresponds to the solvable nilpotent orbit $\mu$. Using Remark \ref{SFotimesSF}, $\sigma_{(\lambda,\emptyset)} \otimes \SF \otimes \SF =a_V\sigma_{(\lambda,\emptyset)} \otimes \bigwedge^{\bullet}(V)$, so by Corollary \ref{enhancedwtype} all constituents of $a_V\sigma_{(\lambda,\emptyset)} \otimes \bigwedge^{\bullet}(V)$ also lie in $\bigoplus_{\sigma\in F_\nu: \mu \ge \nu \ge \mu^t_C} \sigma$. Therefore, all constituents of $a_V^2\sigma_{(\lambda,\emptyset)} \otimes \sigma_{(\lambda,\emptyset)} \otimes (\bigwedge^{\bullet}(V) \otimes \bigwedge^{\bullet}(V) )$ also lie in $(\bigoplus_{\sigma\in F_\nu: \mu \ge \nu \ge \mu^t_C} \sigma)^{\otimes 2}$. We conclude the proof by noting that $\sigma_{(\lambda,\emptyset)} \otimes \sigma_{(\lambda,\emptyset)}$ contains a copy of the trivial representation, similarly to above.
\end{proof}

We provide a more explicit exposition of this result. By Exercise 6.13 of \cite{fulton1991representation}, we have an explicit decomposition 
\[\bigwedge\nolimits^i(V) \otimes \bigwedge\nolimits^j(V)=\bigoplus_{\alpha\in P(i+j)}K_{\alpha,(i,j)}\mathbb{S}_{\alpha^{t}}V\] where $K_{\alpha,(i,j)}$ is the Kostka number, and $\mathbb{S}_{\alpha^{t}}V$ is the Weyl module. Recall that the Kostka number $K_{\alpha, \beta}$ is nonzero iff $\alpha\ge\beta$ in the dominance ordering. Therefore, $\bigwedge^i(V) \otimes \bigwedge^j(V)$ and $\bigoplus_{i+j\ge k\ge i} \mathbb{S}_{(2^{i+j-k},1^{2k-i-j})}V$ have the same constituents. Together with the fact that $\bigwedge^i(V)\otimes \sgn = \bigwedge^{n-i}(V)$, and collecting the constituents for all choices of $(i,j)$, we obtain the following set of irreducible $\GL(V)$-modules: $\{\mathbb{S}_{(2^r,1^s)}V, \ \mathbb{S}_{(2^r,1^s)}V\otimes \sgn : 0\leq2r+s\leq n\}$.

Fix $r,s$ with $0 \leq 2r+s\leq n$. We now decompose the irreducible $\GL(V)$-module $\mathbb{S}_{(2^r,1^s)}V$ into irreducible $\text{O}(V)$-modules of the form $\mathbb{S}_{[(2^{r'},1^{s'})]}V$ as follows. By applying Theorem 19.22 to (25.37) both of \cite{fulton1991representation}, we have the following formula. Given $\lambda=(\lambda_1\ge\cdots\ge\lambda_n\ge 0)$,

\[\text{Res}^{\text{GL}(V)}_{\text{O}(V)} \mathbb{S}_\lambda V=\bigoplus_{\Bar{\lambda}} N_{\lambda\Bar{\lambda}}\mathbb{S}_{[\Bar{\lambda}]}V\]

where we sum over all $\Bar{\lambda}=(\Bar{\lambda_1}\ge\cdots\ge\Bar{\lambda_n}\ge 0)$, where $N_{\lambda\Bar{\lambda}}=\sum_{\delta}N_{\delta\Bar{\lambda}}^\lambda$, and where $N_{\delta\Bar{\lambda}}^\lambda$ are Littlewood-Richardson coefficients. Specialising to $\lambda=(2^r,1^s)$, we see if $N_{\lambda\Bar{\lambda}}>0$ then $\Bar{\lambda}$ must be of the same form $(2^{r'},1^{s'})$ with $ 0\leq 2r'+s'\leq n$. Hence we obtain the following:

\begin{corollary}
 Let $W$ be the Weyl group of $\mathfrak{sp}(2n)$, and $\mu$ be the
 minimal solvable orbit.
Then $(\bigoplus_{\sigma\in F_\nu: \mu \ge \nu \ge \mu^t_C} \sigma)^{\otimes 2} $ contains all the irreducible characters of $W$ occurring in $\mathbb{S}_{[\lambda]} V$, where $\lambda=(2^r,1^s)$ for $ 0\leq 2r+s\leq n$ , and $V$ is the reflection representation of W.
\end{corollary}

We remark on the case when $\mu$ is self-dual:

\begin{corollary}
 Let $n=k(k+1)$ and $W$ be the Weyl group of $\mathfrak{sp}(2n)$. 
 Suppose $\mu=(2k,2k,\cdots,2,2)$ is the unique self-dual minimal solvable orbit.
Then $(\bigoplus_{\sigma\in F_\mu} \sigma)^{\otimes 2} $ contains all the irreducible characters of $W$ occurring in $\mathbb{S}_{[\lambda]} V$, where $\lambda=(2^r,1^s)$ for $ 0\leq 2r+s\leq n$, and $V$ is the reflection representation of W.
\end{corollary}

\subsection{Generalised Saxl Conjecture for Type D}

In this section, we look into Saxl's conjecture for type D. Let $W=W(D_n)$ be the Weyl group of type D, which is an index 2 subgroup of $W(C_n)$. We describe the representations of $W$ in terms of the representations of $W(C_n)$ as follows: 

\begin{proposition}[Proposition 11.4.4 of \cite{carter1993finite}]
The representation $\sigma_{(\alpha,\beta)}$ remains irreducible on restriction to $W(D_n)$ if $\alpha \neq \beta$, and $\sigma_{(\alpha,\beta)},\sigma_{(\beta,\alpha)}$ coincide on restriction to $W(D_n)$. If $\alpha=\beta$, then $\sigma_{(\alpha,\alpha)}$ decomposes into two irreducible characters $\sigma_{(\alpha,\alpha)}',\sigma_{(\alpha,\alpha)}''$. This gives an exact list of irreducible representations of $W(D_n)$.
\end{proposition}

Let $P_1(2n)$ be the partitions of $2n$ in which even parts occur with even multiplicity. The nilpotent orbits in $\mathfrak{so}(2n)$ are parameterised by $P_1(2n)$, except that "very even" partitions (those with only even parts, each having even multiplicity) correspond to two orbits, denoted as $I,II$. The Spaltenstein map $d:P_1(2n) \rightarrow P_1(2n)$ is given by $\nu \mapsto \nu^t_D$, where the subscript $D$ represents the $D$-collapse. It induces order-reversing maps to the corresponding sets of nilpotent orbits, if we decree that the numerals $I$ and $II$ attached to a very even partition in $P_1(2n)$ get interchanged by $d$ if $n$ is odd, or preserved if $n$ is even. It is clear from the definition that an orbit $\nu\in P_{1}(2n)$ is special if and only if $\nu^t\in P_{1}(2n)$.

A nilpotent orbit is solvable if and only if the corresponding partition where all parts are odd and appear with multiplicity at most two. We now recall the explicit construction of irreducible genuine representations of $\Tw$:

We define an equivalence relation $\sim$ on $P(n)$ given by $\lambda\sim \lambda^t$. As before, we also define an equivalence relation $\widetilde{\sigma}\sim \widetilde{\sigma}\otimes \sgn$ on $\Irrg(\Tw)$.

\begin{theorem}[\cite{ReadProjectiveReps}, Chapter 7 and 8]
Let $\Tw$ be the Pin cover of Weyl group of type $D_n$. There is a one-to-one correspondence
     \[
    \Irrg(\Tw) / \sim  \quad \leftrightarrow \quad P(n) / \sim
    \]

More precisely, recall the irreducible representations $\sigma_\lambda$ of $S_n$ and $\widetilde{\sigma}_\lambda$ of $\widetilde{W(C_n)}$. We have:

\begin{enumerate}
    \item if $n$ is odd, every irreducible genuine $\widetilde{W(C_n)}$ representation restricts to a genuine irreducible $\Tw$-representation, and this gives a complete set of inequivalent irreducible genuine representations. Moreover $\widetilde{\sigma}_\lambda \sim \widetilde{\sigma}_{\lambda^t}$.
    
    \item if $n$ is even, and if $\lambda=\lambda^t$, then $\widetilde{\sigma}_\lambda$ restricts to a sum of two associate irreducible $\Tw$-representations, namely $\Tsig_\lambda'$ and $\Tsig_\lambda''$; otherwise, $\widetilde{\sigma}_\lambda$ restricts to an irreducible $\Tw$-representation.
\end{enumerate}
\end{theorem}

Let $\O$ be the unique minimal solvable orbit. Let $\lambda$ be any partition of $n$ corresponding to a genuine representation associated with $\mu$ in the above setting. For a family $F$, denote $O_F$ as the special orbit corresponding to $F$. Whichever the parity, Corollary \ref{enhancedwtype} controls the $W$-constituents of $\Tsig_\lambda \otimes \SF$. By the same discussion of type C, we get an almost identical result for Saxl's conjecture for type D as in last section:

\begin{corollary}
 Let $W$ be the Weyl group of $\mathfrak{so}(2n)$, and $\O$ be the
 minimal solvable orbit.
Then $$\Big(\bigoplus_{\sigma\in F: \ \O \ge \O_F \ge d(\O)} \sigma \Big)^{\otimes 2} $$ contains all of the irreducible characters of $W$ occurring in $\mathbb{S}_{[\lambda]} V$, where $\lambda=(2^r,1^s)$ for $ 0\leq 2r+s\leq n$, and $V$ is the reflection representation of W.
\end{corollary}

\begin{remark}
    We don't attempt to phrase the corollary in terms of the partitions, simply because the closure ordering for type D is a little bit involved. One may refer to \cite{collingwood1993nilpotent} for more details. However, we won't encounter this issue when considering the special case where the minimal solvable orbit is self-dual.
\end{remark}

\begin{corollary}
     Let $n=k^2$ and $W$ be the Weyl group of $\mathfrak{so}(2n)$.
    Suppose the nilpotent orbit $\mu=(2k-1,2k-1,2k-3,2k-3,\cdots,1,1)$ corresponds to Lusztig family $F_\mu$. Then $(\bigoplus_{\sigma\in F_\mu} \sigma)^{\otimes 2} $ contains all the irreducible characters of $W$ occurring in $\mathbb{S}_{[\lambda]} V$, where $\lambda=(2^r,1^s)$ for $ 0\leq 2r+s\leq n$ , and $V$ is the reflection representation of W.
\end{corollary}

\section{A Characterisation of the Cuspidal Families}

In this section, we focus on the case that the minimal solvable orbit $\O$ is self-dual. Our goal is to describe the Lusztig family corresponding to this orbit (via the Springer correspondence) in terms of the Pin cover of $W$. We saw in Remark \ref{cuspidal} that this family is precisely the cuspidal family, except for type $A$.

As we've seen in Corollary \ref{enhancedwtype}, given a self-dual solvable orbit $\O$, and a $\Tw$-type $\widetilde{\sigma}$ corresponding to $\O$, then any $W$-type appearing in $\widetilde{\sigma} \otimes \SF$ must lie in the family $F$ corresponding to $\O$. It's natural to wonder whether we have a converse:

\begin{theorem}
     \label{getallfamily}
     Let $\mathfrak{g}$ be a simple Lie algebra with Weyl group $W$. Suppose $\mathfrak{g}$ has a solvable nilpotent orbit $\O_F$ corresponding to Lusztig family $F$ with the property that $d(\O_F)=\O_F$. Let $\Psi^{-1}(\O_F)$ be the set of all irreducible genuine representations of $\Tw$ corresponding to $\O_F$. Then, the $W$-constituents of $\bigoplus\limits_{\widetilde{\sigma}\in \Psi^{-1}(\O_F)}  (\widetilde{\sigma}\otimes \SF)$ are precisely the $W$-types in family $F$.
\end{theorem}

We immediately get an equivalent formulation of the generalised Saxl conjecture, Ver II: 

\begin{corollary}
    Let $\mathfrak{g}$ be a simple Lie algebra with Weyl group $W$. Suppose $\mathfrak{g}$ has a solvable nilpotent orbit $\O$ such that $d(\O)=\O$. Then the following are equivalent:
    \begin{enumerate}
        \item Conjecture \ref{gen_Saxl_conj_ver_II};
        \item $\bigwedge^{\bullet}(V)\otimes(\bigoplus_{\Tsig\in\Psi^{-1}(\O)} \Tsig)^{\otimes 2} $ contains all irreducible characters of $W$, where $V$ is the reflection representation of $W$.
    \end{enumerate}
\end{corollary}

Note that Theorem \ref{getallfamily} is clearly true for type $A$. Again for demonstration, we also prove it for the exceptional types by direct calculation.

\begin{proposition}
    Theorem \ref{getallfamily} holds for the exceptional types.
\end{proposition}

\begin{proof}
    We verify the statement case-by-case using GAP 3 \cite{GAP3} and package CHEVIE \cite{CHEVIE}. We present the verification for $G_2$ here, using Morris' notation for projective representations of exceptional groups \cite{Morris_Exceptional}. In this case, $\O_F=G_2(a_1)$, and $\Psi^{-1}(\O_F)=\{2_s,2_{ss},2_{sss}\}$.\\

    For character $\phi$ occurring with multiplicity $m$ in $\bigoplus\limits_{\widetilde{\sigma}\in \Psi^{-1}(\O_F)}  (\widetilde{\sigma}\otimes \SF)$, list the data as: $\phi \; (m)$.

\begin{center}
\begin{tabular}{ c c c c }
$\phi_{1,3}' \;(1)$
& $\phi_{1,3}'' \; (1)$
& $\phi_{2,1} \; (1)$
& $\phi_{2,2} \; (2)$
\end{tabular}
\end{center}
    
    The remaining tables can be found in the appendix.
\end{proof}

We now prove Theorem \ref{getallfamily} for types C and D.

\subsection{Proof for Type C}

Let $n=k(k+1)$, and $\mu=(2k,2k,\cdots,2,2)$ be the self-dual solvable orbit.
Let $\lambda=(k^{k+1})$, so $\Psi^{-1}(\mu)=\{\sigma_{(\lambda,\emptyset)}\otimes S\}, \text{or } \{\sigma_{(\lambda,\emptyset)}\otimes S^{\pm}\}$, as in Theorem \ref{Tw_parameterisation}. Let $\Tsig$ be the sum of elements in $\Psi^{-1}(\mu)$, then we have $\widetilde{\sigma}\otimes \SF= a_V\sigma_{(\lambda,\emptyset)}\otimes \bigwedge^\bullet V$ by Remark \ref{SFotimesSF}. Therefore, to prove the statement it suffices to show that the $W$-types in $\sigma_{(\lambda,\emptyset)}\otimes \bigwedge^\bullet V$ exhaust the family containing $\sigma_{(\lambda,\emptyset)}$. We first recall a few important lemmas that reduce the problem into a combinatorial setting.

\begin{lemma}[Lemma 3.4, \cite{CIUBOTARU20221}]
\label{C_tensorwithwedge}
Let $(\alpha,\beta)\in BP(n)$, and $\gamma\in P(n)$. Then $$a_V\big\langle\sigma_{(\alpha,\beta)}, \sigma_{(\gamma,\emptyset)}\otimes \bigwedge\nolimits^\bullet V\big\rangle_{W(C_n)}=N_{\alpha\beta^t}^\gamma$$ where $N_{\alpha\beta^t}^\gamma$ is the Littlewood-Richardson coefficient.
\end{lemma}

\begin{remark}
    The reader may refer to Appendix A.1 of \cite{fulton1991representation} for the rules of calculating the Littlewood-Richardson coefficient, which will be used in the proof of Proposition \ref{positiveRL_C}.
\end{remark}

\begin{lemma}
\label{family_C}
Suppose $(\alpha,\beta)\in BP(n)$, then $\sigma_{(\alpha,\beta)}$ lies in the same family as $\sigma_{(\lambda,\emptyset)}$, where $\lambda=(k^{k+1})$, if and only if the following holds:

\begin{enumerate}
    \item $\alpha$ has $k+1$ components $\alpha_0\leq\alpha_1\leq\cdots\leq\alpha_k$, allowing zeroes.
     \item $\beta$ has $k$ components $\beta_0\leq\beta_1\leq\cdots\leq\beta_{k-1}$, allowing zeroes.
     \item the collection of numbers $\{\Tilde{\alpha}_i=\alpha_i+i,\Tilde{\beta}_j=\beta_j+j:0\leq i\leq k,0\leq j \leq k-1\}$ coincides with the set $\{0,1,\cdots,2k\}$.
\end{enumerate}
\end{lemma}

\begin{proof}
    This is a special case of Proposition 11.4.3 of \cite{carter1993finite}.
\end{proof}

Therefore, we only need to prove the following purely combinatorial statement:

\begin{proposition}
    \label{positiveRL_C}
    If $(\alpha,\beta)\in BP(n)$ satisfies the properties listed in Lemma \ref{family_C}, then $N_{\alpha\beta^t}^\lambda>0$, where $\lambda=(k^{k+1})$.
\end{proposition}

\begin{proof}
In order to show the Littlewood-Richardson coefficient $N_{\alpha\beta^t}^\lambda>0$, it suffices to construct a \textbf{strict} $\beta^t$-expansion from $\alpha$ to obtain $\lambda$. Suppose we start from $\alpha^0=(\alpha_0^0\leq\cdots\leq\alpha_k^0)$ and $\beta^0=(\beta_0^0\leq\cdots\leq\beta_{k-1}^0)$, with $(\alpha^0,\beta^0) \in BP(n)$ satisfying the properties listed in Lemma \ref{family_C}. Then $A^0=\{\alpha_i^0+i
\}$ and $B^0=\{\beta_j^0+j\}$ form a partition of the set $\{0,1,\cdots,2k\}$.

Now let $\alpha^1=(\alpha_0^1\leq\cdots\leq\alpha_k^1)$, where $\alpha_i^1=\alpha_{i+1}^0$ ,where we take convention $\alpha_{k+1}^0=k$. We then let $\beta^1=(\beta_0^1\leq\cdots\leq\beta_{k-1}^1)$, where $\beta_j^1=\text{max}(\beta_j^0-1,0)$ for all $j$. Note these operations corresponds to deleting the first row of the Young diagram of $\beta^t$ and adding new boxes to the Young diagram of $\alpha$ in such a way that we add at most one block at each column. We must check $A^1=\{\alpha_i^1+i
\}$ and $B^1=\{\beta_j^1+j\}$ form a partition of the set $\{0,1,\cdots,2k\}$. Clearly by construction we see all elements in both collections lie in $\{0,1,\cdots,2k\}$, and there are precisely $2k+1$ numbers in total. So it suffices to show there are no repetitions. Clearly there are no repetitions in both collections $A^1$ and $B^1$. Now suppose $\alpha_i^1+i=\beta_j^1+j$, there are several possibilities.
\begin{enumerate}
    \item If $i=k$, then $\beta_j^1+j=2k$. This forces $\beta_{k-1}^1+(k-1)\ge 2k$, which is impossible.
    \item If $\beta_j^0=0$, then $\beta_j^1=0$. It follows that $\alpha_i^1+i=\alpha_{i+1}^0+i=j \implies \alpha_{i+1}^0+(i+1)=j+1$. We thus must have $\beta_{j+1}^0>0$, since otherwise we get $\alpha_{i+1}^0+(i+1)=\beta_{j+1}^0+j+1$, contradiction. 
    However, note that $\beta_l^0+l=l$ for each $l\leq j$, and $\beta_{j+1}^0+j+1\ge j+2$. We know the minimal element in $A^0$ is $\alpha_0^0+0=j+1$. Therefore $\alpha_0^0+0=j+1=\alpha_{i+1}^0+(i+1)$, contradiction.
    \item Otherwise, we have $\alpha_i^1=\alpha_{i+1}^0$ and $\beta_j^1=\beta_j^0-1$, for $i<k$. Therefore $\alpha_{i+1}^0+i=\beta_j^0+j-1$, which is equivalent to $\alpha_{i+1}^0+(i+1)=\beta_j^0+j$, contradiction.
\end{enumerate}

Therefore, $(\alpha^1,\beta^1)$ is another bi-partition satisfying the property 3 listed in Lemma \ref{family_C}. This immediately implies that $(\alpha^1,\beta^1)$ is a bi-partition of $(1+...+2k)-(1+...+k)-(1+...+k-1)=k(k+1)=n$ satisfying the properties in Lemma 6.5 i.e. the number of blocks we deleted from $\beta^0$ is equal to the number of blocks we added to $\alpha^0$. We can thus perform the above step iteratively to get 

\[
(\alpha^0,\beta^0) \xrightarrow{1} (\alpha^1,\beta^1) \xrightarrow{2} \cdots \xrightarrow{m} (\alpha^m,\beta^m).
\] 

where the number on each arrow represents the number we put in each block when we added them. For example, we put a number $x$ in the blocks which we added from $\alpha^{x-1}$ to $\alpha^{x}$.

Eventually we must have $\beta^m=(0^k)$ and thus $\alpha^m=(k^{k+1})$ is by force. It only remains to check this procedure indeed gives a gives a strict expansion of $\alpha=\alpha^0$. But this is obvious since the number of blocks indexed by $x$ in row $(r-1)$ is precisely the number of blocks indexed by $(x+1)$ in row $r$, and the first row only admits number $1$. This concludes the proof.
\end{proof}

\begin{remark}
Although we don't need the following fact, it can be appreciated from the above proof that each character in the family corresponding to $\mu$ appears exactly once in $\sigma_{(\lambda,\emptyset)}\otimes \bigwedge(V)$. Indeed, it's equivalent to show there is only one strict $\beta^t$-expansion from $\alpha$ to get $\lambda$. However, the construction above is the most \say{efficient} way to fill in the boxes but still just satisfies the property of being strict in a \say{minimal} sense. We leave the details to the reader.
\end{remark} 

From the discussion above, we conclude the proof of Theorem \ref{getallfamily} for type C.

\subsection{Proof for Type D}
The proof of Theorem \ref{getallfamily} for type D follows the idea of the proof in type C. Let $n=k^2$ and $\lambda=(k^k)$. As in the case of type C, it suffices to show that the $W$-types in $\sigma_{(\lambda,\emptyset)}\otimes \bigwedge^\bullet V$ exhaust the family containing $(\lambda,\emptyset)$. We have the following lemma:

\begin{lemma}
\label{family_D}

Let $n=k^2$ and $\lambda=(k^k)$. A representation of $W(D_n)$ lies in the same family as $(\lambda,\emptyset)$ must correspond to $(\alpha,\beta)\in BP(n)$ where $\alpha\neq\beta$. Moreover, $\sigma_{(\alpha,\beta)}$ lies in the same family as $\sigma_{(\lambda,\emptyset)}$ if and only if the following holds:

\begin{enumerate}
    \item $\alpha$ has $k$ components $\alpha_0\leq\alpha_1\leq\cdots\leq\alpha_{k-1}$, allowing zeroes.
     \item $\beta$ has $k$ components $\beta_0\leq\beta_1\leq\cdots\leq\beta_{k-1}$, allowing zeroes.
     \item the collection of numbers $\{\Tilde{\alpha}_i=\alpha_i+i,\Tilde{\beta}_j=\beta_j+j:0\leq i\leq k-1,0\leq j \leq k-1\}$ coincides with the set $\{0,1,\cdots,2k-1\}$.
\end{enumerate}
\end{lemma}

\begin{proof}
    This is a special case of Proposition 11.4.4 of \cite{carter1993finite}.
\end{proof}

The same argument as in Proposition \ref{positiveRL_C} gives us a similar combinatorial statement:

\begin{proposition}
    \label{positiveRL_D}
    If $(\alpha,\beta)\in BP(n)$ satisfies the properties listed in Lemma \ref{family_D}, then $N_{\alpha\beta^t}^\lambda>0$, where $\lambda=(k^k)$.
\end{proposition}

Therefore, we know if $(\alpha,\beta)$ lies in the same family as $(\lambda,\emptyset)$, then $\langle\sigma_{(\alpha,\beta)},\sigma_{(\lambda,\emptyset)}\otimes \bigwedge^{\bullet}V\rangle_{W(C_n)}>0$ by Lemma \ref{C_tensorwithwedge}. Hence $\langle\sigma_{(\alpha,\beta)},\sigma_{(\lambda,\emptyset)}\otimes \bigwedge^\bullet V\rangle_{W(D_n)}>0$ follows. This concludes the proof of Theorem \ref{getallfamily} for type D.

\section{Relations Between the Generalised Saxl Conjecture for Different Types}

In this section, we investigate the implications of the regular Saxl conjecture (type A) on the generalised Saxl conjecture for types C and D. Notably, we also relate these conjectures to several versions of some classical \say{decomposition of tensor product} problems in $S_n$.

\begin{lemma}
Assume Saxl's conjecture for type A. Let $n=k(k+1)$ and $W$ be the Weyl group of $\mathfrak{sp}(2n)$.Then $\sigma_{(\lambda,\lambda)}\otimes\sigma_{(\lambda,\lambda)}$
contains all the irreducible characters of the form $(\nu,\emptyset)$ and $(\emptyset,\nu)$, where $\lambda=(k,k-1,\cdots,1)$.
\end{lemma}

\begin{proof}
As in Section 3.7 of \cite{ciubotaru2012spin}, note that $$\sigma_{(\lambda,\lambda)}=\Ind_{S_{n/2} \times S_{n/2} \times (\Z/2\Z)^n}^{W} (\sigma_\lambda \boxtimes \sigma_\lambda \boxtimes (\triv)^{n/2} \boxtimes (\sgn)^{n/2})$$ where $\sigma_\lambda$ represents the staircase representation of $S_{n/2}$. Therefore $\sigma_{(\lambda,\lambda)}\otimes \sigma_{(\lambda,\lambda)}$ equals:
\[\Ind_{S_{n/2} \times S_{n/2} \times (\Z/2\Z)^n}^{W} \{(\sigma_\lambda \boxtimes \sigma_\lambda \boxtimes (\triv)^{n/2} \boxtimes (\sgn)^{n/2}) \otimes(\Res\,\Ind(\sigma_\lambda \boxtimes \sigma_\lambda \boxtimes (\triv)^{n/2} \boxtimes (\sgn)^{n/2}))\}
\]

In particular, all constituents of $\sigma_\lambda \otimes \sigma_\lambda$ also appear in 
\begin{align*}
&\Ind_{S_{n/2} \times S_{n/2} \times (\Z/2\Z)^n}^{W} \{(\sigma_\lambda \boxtimes \sigma_\lambda \boxtimes (\triv)^{n/2} \boxtimes (\sgn)^{n/2}) \otimes (\sigma_\lambda \boxtimes \sigma_\lambda \boxtimes (\triv)^{n/2} \boxtimes (\sgn)^{n/2})\} \\
&=\Ind_{S_{n/2} \times S_{n/2} \times (\Z/2\Z)^n}^{W} \{(\sigma_\lambda\otimes\sigma_\lambda)\boxtimes(\sigma_\lambda\otimes\sigma_\lambda)\boxtimes (\triv)^n\}
\end{align*}

Assuming Saxl's conjecture for type $A$, and looking into the representations of $S_n$ which are factored through, we get all constituents of $\Ind_{S_{n/2}\times S_{n/2}}\sigma_\alpha \boxtimes \sigma_\beta$, for any $\alpha,\beta \in P(n/2)$. Hence we obtain all representations of $S_n$. Finally, noting that $\sigma_{(\lambda,\lambda)}=\sigma_{(\lambda,\lambda)}\otimes\sgn$, we see the set of constituents appearing in the tensor product must be stable under tensoring with $\sgn$. This concludes the proof.
\end{proof}

\begin{remark}
\label{nosaxlconstituents}
Without assuming Saxl's conjecture, we still can find large numbers of irreducible characters of $W$ of the form $(\nu,\emptyset)$, by combining works of Bessenrodt \cite{BessenrodtSaxlConjecture} and considering appropriate Littlewood-Richardson coefficients.
\end{remark}

By a very similar argument, we can prove the following:

\begin{lemma}
Assume Saxl's conjecture for type A. Let $n=k^2$ and $W$ be the Weyl group of $\mathfrak{so}(2n)$.Then $\sigma_{(\lambda,\xi)}\otimes\sigma_{(\lambda,\xi)}$
contains all the irreducible characters of the form $(\nu,\emptyset)$ and $(\emptyset,\nu)$, where $\lambda=(k,k-1,\cdots,1)$ and $\xi=(k-1,\cdots,1)$.
\end{lemma}

The important observation is that $(\lambda,\lambda)$ (resp.$(\lambda,\xi)$) is the unique special character in the family corresponding to the self-dual solvable orbit. Therefore, the above two lemmas imply:

\begin{corollary}
\label{AimpliesCD}
Assume Saxl's conjecture for type A.
Let $\mathfrak{g}$ be a simple Lie algebra of type C or D with Weyl group $W$. Let $\O_F$ be the unique minimal solvable orbit of $\mathfrak{g}$, and suppose that $\O_F=d(\O_F)$.
Then $(\bigoplus_{\sigma\in F} \sigma)^{\otimes 2} $ contains all irreducible characters of $W$ inflated from $S_n \leq W$.
\end{corollary}

On the other hand, we can explicitly calculate the irreducible characters of $W$ inflated from $S_n \leq W$ that occurs in $(\bigoplus_{\sigma\in F} \sigma)^{\otimes 2} $.

\begin{lemma}
\label{2rowyoung}
Let $W$ be Weyl group of type C or D, and $V$ is its reflection representation. Then all irreducible characters of $W$ occurring in $\bigwedge(V)\otimes\bigwedge(V)$
that inflate from $S_n \leq W$ are precisely those corresponding to Young diagrams of at most two rows.
\end{lemma}

\begin{proof}
This is a direct consequence of Lemma \ref{C_tensorwithwedge}.
\end{proof}

\begin{proposition}
\label{SntypesofCD}
Let $\mathfrak{g}$ be a simple Lie algebra of type C or D with Weyl group $W$. Let $\O_F$ be the unique minimal solvable orbit of $\mathfrak{g}$, and suppose that $\O_F=d(\O_F)$.
Then the $W$-types in $(\bigoplus_{\sigma\in F} \sigma)^{\otimes 2}$ inflated from $S_n \leq W$ are precisely those occurring in the following tensor product:

\[\sigma_\lambda\otimes\sigma_\lambda\otimes \Big(\bigoplus_{0\leq j \leq n/2}\sigma_{n-j, j}\Big)\]

where $\lambda=(k^{k+1})$ if $\mathfrak{g}=\mathfrak{sp}(2n),n=k(k+1)$; $\lambda=(k^k)$ if $\mathfrak{g}=\mathfrak{so}(2n),n=k^2$.
\end{proposition}

\begin{proof}
We observe the following fact. Suppose $\sigma$ is an irreducible character of $W$ inflated from $S_n \leq W$, while $\tau$ is not. Then none of the constituents of $\sigma\otimes \tau$ can factor through $S_n$. With this observation, the proposition follows immediately by combining Lemma \ref{2rowyoung} and Theorem \ref{getallfamily}.
\end{proof}

In particular, combining Corollary \ref{AimpliesCD} and Proposition \ref{SntypesofCD}, we get the following nice implication for the representation theory of symmetric groups.

\begin{corollary}
\label{tensorproddecompose}
   The following statements hold:
    \begin{enumerate}
        \item Let $n=k(k+1)$ and $\lambda=(k^{k+1})\in P(n)$. Assume either the actual Saxl conjecture or the generalised Saxl conjecture (Ver II) for type C. Then $\sigma_\lambda\otimes\sigma_\lambda\otimes (\bigoplus_{0\leq j \leq n/2}\sigma_{n-j, j})$ contains all irreducible representations of $S_n$.
        \item Let $n=k^2$ and $\lambda=(k^k)\in P(n)$. Assume either the actual Saxl Conjecture or the generalised Saxl conjecture (Ver II) for type D. Then $\sigma_\lambda\otimes\sigma_\lambda\otimes (\bigoplus_{0\leq j \leq n/2}\sigma_{n-j, j})$ contains all irreducible representations of $S_n$.
    \end{enumerate}
\end{corollary}

\begin{remark}
    Without assuming any version of the generalised Saxl conjecture, one can already identify many constituents of $\sigma_\lambda\otimes\sigma_\lambda\otimes (\bigoplus_{0\leq j \leq n/2}\sigma_{n-j, j})$ in both cases, namely those that can be identified by the method sketched in Remark \ref{nosaxlconstituents}. We also refer the reader to Section 5 of \cite{BessenrodtSaxlConjecture}, where Bessenrodt proposes a similar conjecture regarding the square tensor of a square partition of symmetric groups.
\end{remark}

\section{Generalised Saxl Conjecture for Non-Crystallographic Coxeter Groups}

\subsection{Formulation}

The notion of \say{Lusztig family} is in fact defined over all finite Coxeter groups, not only the crystallographic ones. In the case of non-crystallographic Coxeter groups, $\Irr(W)$ is partitioned into families using Lusztig's \say{a-function}, see Definition 6.5.7 of \cite{Geck_Pfeiffer_Characters}. This gives us hope to state a version of the generalised Saxl conjecture for all Coxeter groups. The difficulty here is that we lack the notion of nilpotent orbit, which stems from the associated Lie algebra structure of the Weyl group. However, we can take a slightly indirect approach as following.\\

Let $W$ be any finite Coxeter group, and $\Tw$ be its Pin cover as before. Recall that in Theorem \ref{generalisedSpringer} we introduced a central element $\Omega_\Tw$ of $\C[\Tw]$. Let $$\Spec(\Omega_\Tw) := \{\Tsig(\Omega_\Tw): \Tsig\in \Irrg(\Tw)\}$$ and note $\Spec(\Omega_\Tw)\subseteq \R_{>0}$ by Theorem 2.3 (1). This set has an ordering as in Remark \ref{lengthofh}. We also set $$\Min(\Omega_\Tw):=\{\Tsig\in \Irrg(\Tw): \Tsig(\Omega_\Tw) \text{ takes minimal value in } \Spec(\Omega_\Tw)\}$$ We in fact have:

\begin{lemma}
Suppose $W$ is a Weyl group associated with Lie algebra $\mathfrak{g}$. Let $\O$ be the minimal solvable orbit in $\mathfrak{g}$. We have $\Min(\Omega_\Tw)=\Psi^{-1}(\O)$.
\end{lemma}

\begin{proof}
This follows immediately from the fact that minimal solvable orbit is unique, Theorem \ref{generalisedSpringer}, and Remark \ref{lengthofh}.
\end{proof}

\begin{definition}
    \label{good}
    Let $W$ be a Coxeter group. We say a family $F$ of $\Irr(W)$ is \textbf{good} if it contains a character that appears as a $W$-constituent of $\Tsig \otimes \SF$, for some $\Tsig \in \Min(\Omega_\Tw)$.
\end{definition}

\begin{remark}
For $W$ a Weyl group of classical type, the decomposition of $\Tsig\otimes\SF$ can be calculated via explicit combinatorial methods. See \cite{ciubotaru2012spin} for more details.
\end{remark}

One can see that this is a natural definition in the sense of \cite{chan2013spin}, where Chan establishes a version of Theorem \ref{generalisedSpringer} for the non-crystallographic groups.

\begin{lemma}
\label{controlgoodfamily}
Suppose $W$ is a Weyl group associated with Lie algebra $\mathfrak{g}$. Let $\O$ be the minimal solvable orbit in $\mathfrak{g}$. If $F$ is a good family, then $\O \ge \O_F \ge d(\O)$.
\end{lemma}

\begin{proof}
This is just rephrasing Corollary \ref{enhancedwtype}.
\end{proof}

We're now ready to restate the generalised Saxl conjecture in this new setting:

\begin{conjecture}[Generalised Saxl Conjecture for Coxeter Groups]
\label{Saxl_3}
Let $W$ be any finite Coxeter group. Then $\left( \ \bigoplus\limits_{\sigma\in F: F \text{ is good}} \sigma\right)^{\otimes 2} $ contains all irreducible characters of $W$.
\end{conjecture}

By Lemma \ref{controlgoodfamily}, we immediately see that Conjecture \ref{Saxl_3} implies Conjecture \ref{gen_Saxl_conj_ver_II}. One can also easily check that all progress made towards Conjecture \ref{gen_Saxl_conj_ver_II} also holds for Conjecture \ref{Saxl_3}. Moreover, when $W$ is crystallographic and the associated Lie algebra has a self-dual solvable orbit, we see that Conjecture \ref{gen_Saxl_conj_ver_II} is equivalent to Conjecture \ref{Saxl_3}.

\begin{example}
Even for symmetric groups, one can already see Conjecture \ref{Saxl_3} is strictly stronger than Conjecture \ref{gen_Saxl_conj_ver_I}. Consider the symmetric group $S_9$ with minimal solvable orbit $\lambda=(432)$. Then $(333)$ corresponds to an orbit between $(432)$ and $(432)^t=(3321)$. However, one can verify that $\{(333)\}$ is not a good family, for example by using Proposition 4.3 of \cite{CIUBOTARU20221}.
\end{example}

We will present a proof for Conjecture \ref{Saxl_3} in the case of non-crystallographic groups.

\subsection{Proof for Dihedral Groups}

Let $W = D_{2n}$. Recall the families of dihedral groups are given by $\{\triv\},\{\sgn\},F_0$, where $F_0$ contains all irreducible representations of $D_{2n}$ not equal to $\triv$ or $\sgn$ \cite{Geck_Pfeiffer_Characters}. For calculation purposes, we adopt the character tables for $\widetilde{D_{2n}}$ in \cite{chan2013spin}. Let $\alpha_1$ and $\alpha_2$ be a choice of simple roots for $D_{2n}$. Note that the action of $z$ as $1$ or $-1$ distinguishes the non-genuine and genuine characters respectively below.

\renewcommand{\arraystretch}{1.3}
\begin{table}[ht]
    \caption{Character Table for $\widetilde{D_{2n}}$  ($n$ odd)}
\begin{center}
    \begin{tabular}{c|c c c c c c c} 
    \hline
    Characters & 1 & $z$ & $f_{\alpha_1}$ & $zf_{\alpha_1}$ & $(f_{\alpha_1}f_{\alpha_2})^k$ & $z(f_{\alpha_1}f_{\alpha_2})^k$ \\
    & & & & & $(k=1,2,\cdots,\frac{n-1}{2})$ & $(k=1,\cdots,\frac{n-1}{2})$ \\
    \hline
    $\triv$ & 1 & 1 & 1 & 1 & 1 & 1 \\
    $\sgn$ & 1 & 1 & $-1$ & $-1$ & 1 & 1 \\
    $\phi_i(i=1,\cdots,\frac{n-1}{2})$ & 2 & 2 & 0 & 0 & 2cos($\frac{2ik\pi}{n}$) & 2cos($\frac{2ik\pi}{n}$)\\
    $\widetilde{\chi}_1$ & $1$ & $-1$ & $-\sqrt{-1}$ & $\sqrt{-1}$ & $(-1)^k$ & $(-1)^{k+1}$ \\
    $\widetilde{\chi}_2$ & $1$ & $-1$ & $\sqrt{-1}$ & $-\sqrt{-1}$ & $(-1)^k$ & $(-1)^{k+1}$ \\
    $\widetilde{\rho_i}(i=1,\cdots,\frac{n-1}{2})$ & $2$ & $-2$ & $0$ & $0$ & $2(-1)^k\text{cos}\frac{2ik\pi}{n}$ & $2(-1)^{k+1}\text{cos}\frac{2ik\pi}{n}$ \\
    \hline
    \end{tabular}
\end{center}
\end{table}
\renewcommand{\arraystretch}{1}

\renewcommand{\arraystretch}{1.3}
\begin{table}[ht]
\caption{Character Table for $\widetilde{D_{2n}}$  ($n$ even)}
\begin{center}
    \begin{tabular}{c|c c c c c c c c} 
    \hline
    Characters & 1 & $z$ & $f_{\alpha_1}$ & $f_{\alpha_2}$ & $(f_{\alpha_1}f_{\alpha_2})^k$ & $z(f_{\alpha_1}f_{\alpha_2})^k$ & $(f_{\alpha_1}f_{\alpha_2})^{n/2}$\\
    & & & & & $(k=1,2,\cdots,\frac{n-2}{2})$ & $(k=1,\cdots,\frac{n-2}{2})$ & \\
    \hline
    $\triv$ & 1 & 1 & 1 & 1 & 1 & 1 & 1\\
    $\sgn$ & 1 & 1 & $-1$ & $-1$ & 1 & 1 & $1$\\
    $\sigma_{\alpha_1}$ & $1$ & $1$ & $-1$ & $1$ & $(-1)^k$ & $(-1)^k$ & $(-1)^{n/2}$ \\
    $\sigma_{\alpha_2}$ & $1$ & $1$ & $1$ & $-1$ & $(-1)^k$ & $(-1)^k$ & $(-1)^{n/2}$ \\
    $\phi_i(i=1,\cdots,\frac{n-2}{2})$ & 2 & 2 & 0 & 0 & 2cos($\frac{2ik\pi}{n}$) & 2cos($\frac{2ik\pi}{n}$) & $2(-1)^i$ \\
    $\widetilde{\rho_i}(i=1,\cdots,\frac{n}{2})$ & $2$ & $-2$ & $0$ & $0$ & $2\text{cos}\frac{(2i-1)k\pi}{n}$ & $-2\text{cos}\frac{(2i-1)k\pi}{n}$ & 0\\
    \hline
    \end{tabular}
\end{center}
\end{table}
\renewcommand{\arraystretch}{1}

\begin{remark}[Remark 2.2 of \cite{chan2013spin}]
The spinor modules  in each case are $\widetilde{\rho}_{\frac{n-1}{2}}$ (when $n$ is odd), and  $\widetilde{\rho}_{\frac{n}{2}}$ (when $n$ is even).
\end{remark}

According to the calculation in \cite{chan2013spin}, we see that 
$\Min(\Omega_{\Tw})=\{\widetilde{\chi}_1,\widetilde{\chi}_2\}$, when $n$ is odd;  $\Min(\Omega_{\Tw})=\{\widetilde{\rho}_1\}$, when $n$ is even.

\begin{remark}
From the explicit description of the set $\Min(\Omega_\Tw)$ in the case of dihedral groups, we see that there is absolutely no hope to generalise Theorem \ref{getallfamily} to non-crystallographic groups in our setting.
\end{remark}

\begin{proposition}
    The dihedral group $D_{2n}$ has only one good family $F_0$.
\end{proposition}

\begin{proof}
Let $S$ be the spinor module for $D_{2n}$. Note $S=S\otimes \sgn$ and $\SF=S$. We only need to show that for any $\Tsig\in\Min(\Omega_\Tw)$, neither $\triv$ nor $\sgn$ can be a constituent of $\Tsig \otimes S$. Suppose not, then either $\triv$ or $\sgn$ is a constituent of $\Tsig \otimes S$, but this forces $\Tsig=S$ by a similar argument as previously. Therefore $S\in\Min(\Omega_\Tw)$, contradiction.
\end{proof}

\begin{proposition}
    Conjecture \ref{Saxl_3} holds for dihedral groups.
\end{proposition}

\begin{proof}
For convenience, we take the convention that $\phi_{i}=\phi_{-i}=\phi_{i+n}$, $\phi_0=\triv\bigoplus\sgn$, $ \phi_{n/2}=\sigma_{\alpha_1}\bigoplus\sigma_{\alpha_2}$, where the last convention is only for $n$ even. Under these conventions, we always have $\phi_i\otimes\phi_j=\phi_{i+j}\bigoplus\phi_{i-j}$ regardless of the parity of $n$. In particular, $\phi_1\otimes\phi_1$ contains $\triv$ and $\sgn$, $\phi_{i-1}\otimes\phi_1$ contains $\phi_i$ for $i>1$, $\phi_2\otimes\phi_1$ contains $\phi_1$, and $\phi_1\otimes\phi_{\frac{n-2}{2}}$ contains $\sigma_{\alpha_1},\sigma_{\alpha_2}$ when $n$ is even. This concludes the proof.
\end{proof}

\subsection{Proof for \texorpdfstring{$H_3 + H_4$}{}}

In this subsection, we follow Geck and Pfeiffer's notation for the characters of $H_3$ and $H_4$ \cite{Geck_Pfeiffer_Characters}. Let $\phi_{x,y}$ be the irreducible representation of dimension $x$ with $b$-invariant $y$. The $b$-invariant is defined in Section 5.2.2 and can be computed by Algorithm 5.3.5, both of \cite{Geck_Pfeiffer_Characters}.

First, we look at $H_{3}$. Recall the Weyl group for $H_3$ is isomorphic to $\Z/2\Z\times A_5$, where $A_5$ is the alternating group on 5 letters. Let $\alpha_i$  $(i=1,2,3)$ be the simple roots, and $w_0$ the non-trivial element in $\Z/2\Z$ (this is actually the longest element in $W(H_3)$). Therefore, to describe the characters of $W(H_3)$ and $\widetilde{W(H_3)}$, it's sufficient to describe the characters of $A_5$ and $\widetilde{A_5}$ respectively. 

More precisely, given a representation $\sigma\in\Irr(A_5)$ (resp. $\Tsig\in\Irrg(\widetilde{A_5})$), we can obtain two representations $\sigma^+,\sigma^-\in\Irr(W(H_3))$ (resp.  $\Tsig^+,\Tsig^-\in\Irrg(\widetilde{W(H_3)})$. We thus present the character tables for $A_5$ and $\widetilde{A_5}$ adapted from \cite{read1974linear}. Set $\tau=(1+\sqrt{5})/2$ and $\Bar{\tau}=(1-\sqrt{5})/2$.

\renewcommand{\arraystretch}{1.3}
\begin{table}[ht]
    \caption{Character Table for $A_5$}
\begin{center}
    \begin{tabular}{c|c c c c c c } 
    \hline
    Characters & 1 & $s_{\alpha_1}s_{\alpha_2}$ & $(s_{\alpha_1}s_{\alpha_2})^2$ & $s_{\alpha_2}s_{\alpha_3}$ & $s_{\alpha_1}s_{\alpha_3}$  \\
    \hline
    $\phi_1$ & 1 & 1 & 1 & 1 & 1
    \\
    $\phi_4$ & 4 & $-1$ & $-1$ & $1$ & 0 \\
    $\phi_3$ & 3 & $\tau$ & $\Bar{\tau}$ & 0 & $-1$ \\
    $\Bar{\phi_3}$ & 3 & $\Bar{\tau}$ & $\tau$ & 0 & $-1$\\
    $\phi_5$ & 5 & 0 & 0 & $-1$ & 1 \\
    \hline
    \end{tabular}
\end{center}
\end{table}
\renewcommand{\arraystretch}{1.3}
\begin{table}[ht]
    \caption{Characters of the Genuine Representations of $\widetilde{A_5}$}
\begin{center}
    \begin{tabular}{c|c c c c c c } 
    \hline
    Characters & $\pm1$ & $\pm f_{\alpha_1}f_{\alpha_2}$ & $\pm(f_{\alpha_1}f_{\alpha_2})^2$ & $\pm f_{\alpha_2}f_{\alpha_3}$ & $\pm f_{\alpha_1}f_{\alpha_3}$  \\
    \hline
    $\widetilde{\chi_2}$ & $\pm 2$ & $\pm\Bar{\tau}$ & $\mp\tau$ & $\pm1$ & 0\\
    $\widetilde{\Bar{\chi_2}}$ & $\pm 2$ & $\pm\tau$ & $\mp\Bar{\tau}$ & $\pm1$ & 0 \\
    $\widetilde{\chi_6}$ & $\pm 6$ & $\mp1$ & $\pm1$ & $0$ & 0\\
    $\widetilde{\chi_4}$ & $\pm 4$ & $\pm1$ & $\mp1$ & $\mp1$ & 0\\
    \hline
    \end{tabular}
\end{center}
\end{table}
\renewcommand{\arraystretch}{1}

Again, from the calculations in \cite{chan2013spin}, we see that the spinor modules correspond to $\widetilde{\Bar{\chi}_2}^{\pm}$, and $\Min(\Omega_\Tw)=\{\widetilde{\chi_2}^{\pm}\}$.

\begin{lemma}
    The group $W(H_3)$ has only one good family $\{\phi_{4,3},\phi_{4,4}\}$.
\end{lemma}

\begin{proof}
We calculate that $(\widetilde{\chi_2}^++\widetilde{\chi_2}^-)\otimes(\widetilde{\Bar{\chi_2}}^{+}+\widetilde{\Bar{\chi_2}}^{-})=2(\phi_4^++\phi_4^-)$. Note that in the notation of \cite{Geck_Pfeiffer_Characters}, $\phi_4^{\pm}$ are precisely the characters $\phi_{4,3},\phi_{4,4}$ and they form a family.
\end{proof}

\begin{proposition}
    Conjecture \ref{Saxl_3} holds for $H_3$. Namely, $(\phi_{4,3}+\phi_{4,4})^{\otimes2}$ contains all irreducible characters of $W(H_3)$, each with multiplicity 2.
\end{proposition}
\begin{proof}
This can be verified directly with the assistance of GAP 3 \cite{GAP3} and CHEVIE \cite{CHEVIE}.
\end{proof}

Next, we deal with $H_4$. Since the characters for $W(H_4)$ and its double cover are quite lengthy, we refer the reader to Table II(i)(ii)(iii) of \cite{read1974linear} for more details. In addition, for consistency of notions, we shall denote the 20 genuine representations as $\widetilde{\chi}_i$ instead of $\chi_i$ in \cite{read1974linear}, for $35\leq i\leq 54$.

Again by \cite{chan2013spin}, we see that the unique spinor module is $\widetilde{\chi}_{36}$ and $\Min(\Omega_\Tw)=\{\widetilde{\chi}_{35}\}$.

\begin{lemma}
 The group $W(H_4)$ has only one good family with $a$-value $6$.
\end{lemma}

\begin{proof}
This is another straightforward calculation. It's easy to verify $\widetilde{\chi}_{35}\otimes \widetilde{\chi}_{36}=\chi_{19}\bigoplus\chi_{29}$, by Table II(i)(ii) of \cite{read1974linear}. But both $\chi_{19}$ and $\chi_{29}$ lies in the family with $a$-value $6$, see Appendix C of \cite{Geck_Pfeiffer_Characters}.
\end{proof}
\begin{proposition}
    Conjecture \ref{Saxl_3} holds for $H_4$. Namely, if $F$ is the family of $a$-value 6. Then $(\bigoplus_{\sigma \in F}\sigma)^{\otimes 2}$ contains all irreducible characters of $W(H_4)$.
\end{proposition}

\begin{proof}
We directly verify this statement using GAP 3 \cite{GAP3} and CHEVIE \cite{CHEVIE}. $W(H_4)$ has 34 characters, and the family with $a$-value 6 consists of: $$\{
\phi_{24,6},
\phi_{24,7},
\phi_{40,8},
\phi_{48,9},
\phi_{18,10},
\phi_{30,10}',
\phi_{30,10}'',
\phi_{24,11}, 
\phi_{16,11}, 
\phi_{6,12}, 
\phi_{8,12},
\phi_{10,12},
\phi_{24,12}, 
\phi_{16,13},
\phi_{8,13},
\phi_{6,20}
\}$$

We summarize the constituents and multiplicities in $(\bigoplus_{\sigma \in F}\sigma)^{\otimes 2}$ in the following table:\renewcommand{\arraystretch}{1.2}
\begin{center}
\begin{tabular}{ l l l l l l l }
$\phi_{1,0}\ (16)$
& $\phi_{1,60}\ (16)  $
& $\phi_{4,1}\ (40)  $
& $\phi_{4,31}\ (40)  $
& $\phi_{4,7}\ (40)  $
& $\phi_{4,37}\ (40)  $
& $\phi_{6,12}\ (66)  $
\\
$\phi_{6,20}\ (66)  $ &
$\phi_{8,12}\ (83)  $
& $\phi_{8,13}\ (74)  $
& $\phi_{9,2}\ (78)  $
& $\phi_{9,22}\ (78)  $
& $\phi_{9,6}\ (78)  $
& $\phi_{9,26}\ (78)  $
\\
$\phi_{10,12}\ (104)  $
& $\phi_{16,11}\ (136)  $
&
$\phi_{16,13}\ (136)  $
& $\phi_{16,3}\ (122)  $
& $\phi_{16,21}\ (122)  $
& $\phi_{16,6}\ (126)  $
& $\phi_{16,18}\ (126)  $
\\
$\phi_{18,10}\ (150)  $ 
& $\phi_{24,11}\ (194)  $
& $\phi_{24,7}\ (194)  $
&
$\phi_{24,12}\ (193)  $
& $\phi_{24,6}\ (193)  $
& $\phi_{25,4}\ (186) $
&
$ \phi_{25,16}\ (186) $
\\
$ \phi_{30,10}'\ (236)  $
& $ \phi_{30,10}''\ (236)  $
& $ \phi_{36,5}\ (250)  $
& $ \phi_{36,15}\ (250)  $
&
$ \phi_{40,8}\ (303)  $
& $ \phi_{48,9}\ (344)  $ \\
\end{tabular}
\end{center}
\ \\ This concludes the proof.
\end{proof}

We thus obtain the generalised Saxl conjecture for non-crystallographic groups: 
\begin{theorem}
Let $W$ be any non-crystallographic group. Then $\left( \ \bigoplus\limits_{\sigma\in F: F \text{ is good}} \sigma\right)^{\otimes 2}$ contains all irreducible characters of $W$.
\end{theorem}

As a final remark, we demonstrate the connection between the cuspidal family and the notion of good families:

\begin{proposition}
Let $W$ be a finite Coxeter group. Suppose $W$ has a cuspidal family $F$, then $F$ is the unique good family for $W$.
\end{proposition}

\begin{proof}
If $W$ is crystallographic, this follows directly from Remark \ref{cuspidal} and Lemma \ref{controlgoodfamily}. If $W=D_{2n}$ is dihedral, then there is a unique cuspidal family $F_0$ by Lemma 8.5 of \cite{bellamy2016cuspidal}. Finally, $\Tilde{j}$-induction for $H_3,H_4$ can be made explicit through tables D.1 and D.2 of \cite{Geck_Pfeiffer_Characters}, and we see that there are no cuspidal families for $H_3$ and $H_4$.
\end{proof}

\newpage
\appendix

\section{Appendix}

\subsection{Generalised Saxl Conjecture: Tables for Exceptional Types}

In this appendix for Conjecture \ref{gen_Saxl_conj_ver_II}, we show the calculations done by GAP 3 \cite{GAP3} and package CHEVIE \cite{CHEVIE} to prove it for the exceptional types. We use Carter's notation \cite{carter1993finite} for the Weyl groups, and Geck and Pfieffer's notation \cite{Geck_Pfeiffer_Characters} for $H_{3}$ and $H_{4}$.

\begin{proposition}[Generalised Saxl Conjecture for Exceptional Groups]
Let $\mathfrak{g}$ be an exceptional Lie algebra with Weyl group $W$. Let $\O_F$ be the unique solvable orbit of self-dual.
    
Then $(\bigoplus_{\sigma\in F} \sigma)^{\otimes 2}$ contains all irreducible characters of $W$.
\end{proposition}

For character $\phi$ occurring with multiplicity $m$ in $(\bigoplus_{\sigma\in F} \sigma)^{\otimes 2}$ we list the data as: $\phi \; (m)$.\\

\par\noindent\rule{\textwidth}{0.2pt}\\

\renewcommand{\arraystretch}{1.2}

$G_2$ (6 characters), $\O_F=G_2(a_1)$, $F = \{ 
\phi_{1,3}', \;
\phi_{1,3}'', \;
\phi_{2,1}, \;
\phi_{2,2}
\}$

\begin{center}
\begin{tabular}{ c c c c c c }
$\phi_{1,0} \;(4)$
& $\phi_{1,6} \; (4)$
& $\phi_{1,3}' \; (2)$
& $\phi_{1,3}'' \; (2)$
& $\phi_{2,1} \; (6)$
& $\phi_{2,2} \; (6)$
\end{tabular}
\end{center}

\par\noindent\rule{\textwidth}{0.2pt}\\

$F_4$ (25 characters), $\O_F=F_4(a_3)$,

$F = \{
\phi_{12,4},\
\phi_{16,5}, \
\phi_{6,6}', \
\phi_{6,6}'', \
\phi_{9,6}', \
\phi_{9,6}'', \
\phi_{4,7}', \
\phi_{4,7}'', \ 
\phi_{4,8},  \
\phi_{1,12}', \ 
\phi_{1,12}''
\}$

\begin{center}
\begin{tabular}{ l l l l l l l l }
$\phi_{1,0}\ (11)$
& $\phi_{1,12}''\ (5)  $
& $\phi_{1,12}'\ (5)  $
& $\phi_{1,24}\ (11)  $
& $\phi_{2,4}''\ (9)  $
& $\phi_{2,16}'\ (9)  $
& $\phi_{2,4}'\ (9)  $
& $\phi_{2,16}''\ (9)  $
\\
$\phi_{4,8}\ (25)  $
& $\phi_{9,2}\ (46)  $
& $\phi_{9,6}''\ (40)  $
& $\phi_{9,6}'\ (40)  $
& $\phi_{9,10}\ (46)  $
& $\phi_{6,6}'\ (35)  $
& $\phi_{6,6}''\ (39)  $
& $\phi_{12,4}\ (57)  $
\\
$\phi_{4,1}\ (24)  $
& $\phi_{4,7}''\ (16)  $
& $\phi_{4,7}'\ (16)  $
& $\phi_{4,13}\ (24)  $
& $\phi_{8,3}''\ (28)  $
& $\phi_{8,9}'\ (28)  $
& $\phi_{8,3}'\ (28)  $
& $\phi_{8,9}''\ (28)  $
\\
$\phi_{16,5}\ (68)  $

\end{tabular}
\end{center}

\par\noindent\rule{\textwidth}{0.2pt}\\

$E_6$ (25 characters), $\O_F=D_4(a_1)$, $F = \{
\phi_{80,7},\;
\phi_{60,8},\;
\phi_{90,8},\;
\phi_{10,9},\;
\phi_{20,10}
\}$

\begin{center}
\begin{tabular}{ l l l l l l l }
$\phi_{1,0}\ (5)  $
& $\phi_{1,36}\ (5)  $
& $\phi_{10,9}\ (16)  $
& $\phi_{6,1}\ (12)  $
& $\phi_{6,25}\ (12) $
& $\phi_{20,10}\ (34)  $
& $\phi_{15,5}\ (27)  $
\\
$\phi_{15,17}\ (27)  $
& $\phi_{15,4}\ (21)  $
& $\phi_{15,16}\ (21)  $
& $\phi_{20,2}\ (28)  $
& $\phi_{20,20}\ (28)  $
& $\phi_{24,6}\ (36)  $
& $\phi_{24,12}\ (36)  $
\\
$\phi_{30,3}\ (38)  $
& $\phi_{30,15}\ (38)  $
& $\phi_{60,8}\ (81)  $
& $\phi_{80,7}\ (101)  $
& $\phi_{90,8}\ (119)  $
& $\phi_{60,5}\ (77)  $
& $\phi_{60,11}\ (77)  $
\\
$\phi_{64,4}\ (83)  $
& $\phi_{64,13}\ (83)  $
& $\phi_{81,6}\ (102)  $
& $\phi_{81,10}\ (102)  $

\end{tabular}
\end{center}

\newpage

\par\noindent\rule{\textwidth}{0.2pt}\\

\renewcommand{\arraystretch}{1.1}

$E_7$ (60 characters), $\O_F=A_4+A_1$, $F = \{
\phi_{512,11},\;
\phi_{512,12}
\}$

\begin{center}
\begin{tabular}{ l l l l l l }
$\phi_{1,0}\ (2)  $
& $\phi_{1,63}\ (2)  $
& $\phi_{7,46}\ (4)  $
& $\phi_{7,1}\ (4)  $
& $\phi_{15,28}\ (6)  $
& $\phi_{15,7}\ (6)  $
\\
$\phi_{21,6}\ (10)  $
& $\phi_{21,33}\ (10)  $
& $\phi_{21,36}\ (10)  $
& $\phi_{21,3}\ (10)  $
& $\phi_{27,2}\ (12)  $
& $\phi_{27,37}\ (12)  $
\\
$\phi_{35,22}\ (14)  $
& $\phi_{35,13}\ (14)  $
& $\phi_{35,4}\ (14)  $
& $\phi_{35,31}\ (14)  $
& $\phi_{56,30}\ (24)  $
& $\phi_{56,3}\ (24)  $
\\
$\phi_{70,18}\ (26)  $
& $\phi_{70,9}\ (26)  $
& $\phi_{84,12}\ (30)  $
& $\phi_{84,15}\ (30)  $
& $\phi_{105,26}\ (40)  $
& $\phi_{105,5}\ (40)  $
\\
$\phi_{105,6}\ (40)  $
& $\phi_{105,21}\ (40)  $
& $\phi_{105,12}\ (40)  $
& $\phi_{105,15}\ (40)  $
& $\phi_{120,4}\ (46)  $
& $\phi_{120,25}\ (46)  $
\\
$\phi_{168,6}\ (62)  $
& $\phi_{168,21}\ (62)  $
& $\phi_{189,10}\ (70)  $
& $\phi_{189,17}\ (70)  $
& $\phi_{189,22}\ (70)  $
& $\phi_{189,5}\ (70)  $
\\
$\phi_{189,20}\ (70)  $
& $\phi_{189,7}\ (70)  $
& $\phi_{210,6}\ (80)  $
& $\phi_{210,21}\ (80)  $
& $\phi_{210,10}\ (72)  $
& $\phi_{210,13}\ (72)  $
\\
$\phi_{216,16}\ (76)  $
& $\phi_{216,9}\ (76)  $
& $\phi_{280,18}\ (106)  $
& $\phi_{280,9}\ (106)  $
& $\phi_{280,8}\ (98)  $
& $\phi_{280,17}\ (98)  $
\\
$\phi_{315,16}\ (112)  $
& $\phi_{315,7}\ (112)  $
& $\phi_{336,14}\ (122)  $
& $\phi_{336,11}\ (122)  $
& $\phi_{378,14}\ (134)  $
& $\phi_{378,9}\ (134)  $
\\
$\phi_{405,8}\ (146)  $
& $\phi_{405,15}\ (146)  $
& $\phi_{420,10}\ (152)  $
& $\phi_{420,13}\ (152)  $
& $\phi_{512,12}\ (182) $
& $\phi_{512,11}\ (182)  $
\end{tabular}
\end{center}

\par\noindent\rule{\textwidth}{0.2pt}\\

$E_8$ (112 characters), $\O_F=E_8(a_7)$,

$F = \{
\phi_{4480,16},\;
\phi_{7168,17},\;
\phi_{3150,18},\;
\phi_{4200,18},\;
\phi_{4536,18},\;
\phi_{5670,18},\;
\phi_{1344,19},\;
\phi_{2016,19},\;
\phi_{5600,19},\;\\
\phi_{2688,20},\;
\phi_{420,20},\;
\phi_{1134,20},\;
\phi_{1400,20},\;
\phi_{1680,22},\;
\phi_{168,24},\;
\phi_{448,25},\;
\phi_{70,32}
\}$

\begin{center}
\begin{tabular}{ l l l l l l }
$\phi_{1,0}\ (17)  $
& $\phi_{1,120}\ (17)  $
& $\phi_{28,8}\ (154)  $
& $\phi_{28,68}\ (154)  $
& $\phi_{35,2}\ (161)  $
\\
$\phi_{35,74}\ (161)  $
& $\phi_{70,32}\ (328)  $
& $\phi_{50,8}\ (196)  $
& $\phi_{50,56}\ (196)  $
& $\phi_{84,4}\ (318)  $
\\
$\phi_{84,64}\ (318)  $
& $\phi_{168,24}\ (635)  $
& $\phi_{175,12}\ (587)  $
& $\phi_{175,36}\ (587)  $
& $\phi_{210,4}\ (747)  $
\\
$\phi_{210,52}\ (747)  $
& $\phi_{420,20}\ (1480)  $
& $\phi_{300,8}\ (1135)  $
& $\phi_{300,44}\ (1135)  $
& $\phi_{350,14}\ (1298)  $
\\
$\phi_{350,38}\ (1298)  $
& $\phi_{525,12}\ (1807)  $
& $\phi_{525,36}\ (1807)  $
& $\phi_{567,6}\ (1986)  $
& $\phi_{567,46}\ (1986)  $
\\
$\phi_{1134,20}\ (3974)  $
& $\phi_{700,16}\ (2301)  $
& $\phi_{700,28}\ (2301)  $
& $\phi_{700,6}\ (2390)  $
& $\phi_{700,42}\ (2390)  $
\\
$\phi_{1400,20}\ (4777)  $
& $\phi_{840,14}\ (2821)  $
& $\phi_{840,26}\ (2821)  $
& $\phi_{1680,22}\ (5642)  $
& $\phi_{972,12}\ (3282)  $
\\
$\phi_{972,32}\ (3282)  $
& $\phi_{1050,10}\ (3487)  $
& $\phi_{1050,34}\ (3487)  $
& $\phi_{2100,20}\ (6970)  $
& $\phi_{1344,8}\ (4517)  $
\\
$\phi_{1344,38}\ (4517)  $
& $\phi_{2688,20}\ (9009)  $
& $\phi_{1400,8}\ (4769)  $
& $\phi_{1400,32}\ (4769)  $
& $\phi_{1575,10}\ (5254)  $
\\
$\phi_{1575,34}\ (5254)  $
& $\phi_{3150,18}\ (10464)  $
& $\phi_{2100,16}\ (6940)  $
& $\phi_{2100,28}\ (6940)  $
& $\phi_{4200,18}\ (13833)  $
\\
$\phi_{2240,10}\ (7353)  $
& $\phi_{2240,28}\ (7353)  $
& $\phi_{4480,16}\ (14695)  $
& $\phi_{2268,10}\ (7514)  $
& $\phi_{2268,30}\ (7514)  $
\\
$\phi_{4536,18}\ (15009)  $
& $\phi_{2835,14}\ (9270)  $
& $\phi_{2835,22}\ (9270)  $
& $\phi_{5670,18}\ (18566)  $
& $\phi_{3200,16}\ (10482)  $
\\
$\phi_{3200,22}\ (10482)  $
& $\phi_{4096,12}\ (13688)  $
& $\phi_{4096,26}\ (13688)  $
& $\phi_{4200,12}\ (13836)  $
& $\phi_{4200,24}\ (13836)  $
\\
$\phi_{6075,14}\ (19966)  $
& $\phi_{6075,22}\ (19966)  $
& $\phi_{8,1}\ (52)  $
& $\phi_{8,91}\ (52)  $
& $\phi_{56,19}\ (236)  $
\\
$\phi_{56,49}\ (236)  $
& $\phi_{112,3}\ (360)  $
& $\phi_{112,63}\ (360)  $
& $\phi_{160,7}\ (562)  $
& $\phi_{160,55}\ (562)  $
\\
$\phi_{448,25}\ (1436)  $
& $\phi_{400,7}\ (1174)  $
& $\phi_{400,43}\ (1174)  $
& $\phi_{448,9}\ (1260)  $
& $\phi_{448,39}\ (1260)  $
\\
$\phi_{560,5}\ (1642)  $
& $\phi_{560,47}\ (1642)  $
& $\phi_{1344,19}\ (3912)  $
& $\phi_{840,13}\ (2540)  $
& $\phi_{840,31}\ (2540)  $
\\
$\phi_{1008,9}\ (2894)  $
& $\phi_{1008,39}\ (2894)  $
& $\phi_{2016,19}\ (5698)  $
& $\phi_{1296,13}\ (3772)  $
& $\phi_{1296,33}\ (3772)  $
\\
$\phi_{1400,11}\ (3886)  $
& $\phi_{1400,29}\ (3886)  $
& $\phi_{1400,7}\ (4098)  $
& $\phi_{1400,37}\ (4098)  $
& $\phi_{2400,17}\ (6748)  $
\\
$\phi_{2400,23}\ (6748)  $
& $\phi_{2800,13}\ (8010)  $
& $\phi_{2800,25}\ (8010)  $
& $\phi_{5600,19}\ (15946)  $
& $\phi_{3240,9}\ (9184)  $
\\
$\phi_{3240,31}\ (9184)  $
& $\phi_{3360,13}\ (9524)  $
& $\phi_{3360,25}\ (9524)  $
& $\phi_{7168,17}\ (20036)  $
& $\phi_{4096,11}\ (11496)  $
\\
$\phi_{4096,27}\ (11496)  $
& $\phi_{4200,15}\ (11676)  $
& $\phi_{4200,21}\ (11676)  $
& $\phi_{4536,13}\ (12728)  $
& $\phi_{4536,23}\ (12728)  $
\\
$\phi_{5600,15}\ (15710)  $
& $\phi_{5600,21}\ (15710)  $
\end{tabular}
\end{center}

\subsection{Characterisation of Cuspidal Families: Tables for Exceptional Types}

In this appendix for Theorem \ref{getallfamily}, we show the calculations done by GAP 3 \cite{GAP3} and package CHEVIE \cite{CHEVIE} to prove it for the exceptional types. We use Carter's notation \cite{carter1993finite} for actual representations and Morris' notation \cite{Morris_Exceptional} for projective representations.

\begin{proposition}
Let $\mathfrak{g}$ be an exceptional Lie algebra with Weyl group $W$. Suppose $\mathfrak{g}$ has a solvable nilpotent orbit $\O_F$ corresponding to Lusztig family $F$ with the property that $d(\O_F)=\O_F$. Let $\Psi^{-1}(\O_F)$ be the set of all irreducible genuine representations of $\Tw$ corresponding to $\O_F$. Then, the $W$-constituents of $\bigoplus\limits_{\widetilde{\sigma}\in \Psi^{-1}(\O_F)}  (\widetilde{\sigma}\otimes \SF)$ are precisely the $W$-types in family $F$.
\end{proposition}

For character $\phi$ occurring with multiplicity $m$ in $\bigoplus\limits_{\widetilde{\sigma}\in \Psi^{-1}(\O_F)}  (\widetilde{\sigma}\otimes \SF)$, we list the data as: $\phi \; (m)$. We omit the characters that don't appear in the decomposition.\\

\par\noindent\rule{\textwidth}{0.2pt}\\

\renewcommand{\arraystretch}{1.2}

$G_2$, $\O_F=G_2(a_1)$, $\Psi^{-1}(\O_F)=\{2_s,2_{ss},2_{sss}\}$

\begin{center}
\begin{tabular}{ c c c c }
$\phi_{1,3}' \;(1)$
& $\phi_{1,3}'' \; (1)$
& $\phi_{2,1} \; (1)$
& $\phi_{2,2} \; (2)$
\end{tabular}
\end{center}

\par\noindent\rule{\textwidth}{0.2pt}\\

\renewcommand{\arraystretch}{1.2}

$F_4$, $\O_F=F_4(a_3)$, $\Psi^{-1}(\O_F)=\{8_{ss},12_{ss},8_s,4_{ss}\}$ 

\begin{center}
\begin{tabular}{ c c c c c c}
$\phi_{12,4} \;(1)$
& $\phi_{16,5} \; (3)$
& $\phi_{6,6}' \; (3)$
& $\phi_{6,6}'' \; (1)$
& $\phi_{9,6}' \; (1)$
& $\phi_{9,6}'' \; (1)$ \\
$\phi_{4,7}' \; (2)$
& $\phi_{4,7}'' \; (2)$
& $\phi_{4,8} \; (2)$
& $\phi_{1,12}' \; (1)$
& $\phi_{1,12}'' \; (1)$
\end{tabular}
\end{center}

\par\noindent\rule{\textwidth}{0.2pt}\\

\renewcommand{\arraystretch}{1.2}

$E_6$, $\O_F=D_4(a_1)$, $\Psi^{-1}(\O_F)=\{40_{ss},20_s\}$

\begin{center}
\begin{tabular}{ c c c c c}
$\phi_{80,7} \;(1)$
& $\phi_{60,8} \; (3)$
& $\phi_{90,8} \; (2)$
& $\phi_{10,9} \; (2)$
& $\phi_{20,10} \; (1)$
\end{tabular}
\end{center}

\par\noindent\rule{\textwidth}{0.2pt}\\

\renewcommand{\arraystretch}{1.1}

$E_7$, $\O_F=A_4+A_1$, $\Psi^{-1}(\O_F)=\{60_s,60_{ss}\}$ 

\begin{center}
\begin{tabular}{ c c }
$\phi_{512,11} \;(2)$
& $\phi_{512,12} \; (2)$
\end{tabular}
\end{center}

\par\noindent\rule{\textwidth}{0.2pt}\\

\renewcommand{\arraystretch}{1.1}

$E_8$, $\O_F=E_8(a_7)$, $\Psi^{-1}(\O_F)=\{224_s,1344_s,1120_s,2016_{ss},2016_{sss},896_s\}$ 

\begin{center}
\begin{tabular}{ c c c c c c}
$\phi_{70,32} \;(1)$
& $\phi_{168,24} \; (3)$
& $\phi_{420,20} \; (5)$
& $\phi_{1134,20} \; (4)$
& $\phi_{1400,20} \; (1)$
& $\phi_{1680,22} \; (1)$ \\
$\phi_{2688,20} \; (4)$
& $\phi_{3150,18} \; (4)$
& $\phi_{4200,18} \; (3)$
& $\phi_{4460,16} \; (1)$
& $\phi_{4536,18} \; (1)$
& $\phi_{5670,18} \; (1)$ \\
$\phi_{448,25} \;(2)$
& $\phi_{1344,19} \; (6)$
& $\phi_{2016,19} \; (4)$
& $\phi_{5600,19} \; (4)$
& $\phi_{7168,17} \; (3)$
\end{tabular}
\end{center}

\newpage

\printbibliography

\end{document}